\documentclass{amsart}

\DeclareMathSymbol{\twoheadrightarrow}  {\mathrel}{AMSa}{"10}

\def\A{{\mathbf A}}
\def\Q{{\mathbb Q}}
\def\Z{{\mathbb Z}}
\def\C{{\mathbb C}}

\def\F{{\mathbb F}}

\def\i{{i}}

\def\Sn{{\mathbf S}_n}

\def\RR{{\mathfrak R}}

\def\Perm{\mathrm{Perm}}

\def\Gal{\mathrm{Gal}}

\def\End{\mathrm{End}}
\def\Aut{\mathrm{Aut}}
                                                      \def\Alt{\mathrm{Alt}}

                                                                            \def\ord{\mathrm{ord}}
\def\Hom{\mathrm{Hom}}

       \def\Lie{\mathrm{Lie}}
       \def\sep{\mathrm{sep}}

\def\I{{\mathcal I}}

\def\ST{{\mathbf S}}
                                                     \def\ol{{\mathbf o}}

\def\fchar{\mathrm{char}}

\def\pr{\mathrm{pr}}
\def\M{\mathrm{M}}

                                                    \def\E{{\mathcal E}}

\def\dim{\mathrm{dim}}

                                            \def\rk{\mathrm{rk}}
\def\OC{{\mathcal O}}

\newtheorem{thm}{Theorem}[section]
\newtheorem{lem}[thm]{Lemma}
\newtheorem{cor}[thm]{Corollary}

\theoremstyle{definition}

\newtheorem{ex}[thm]{Example}
\newtheorem{rem}[thm]{Remark}

\newtheorem{sect}[thm]{}

\title[Endomorphism Algebras of Abelian varieties]
{Endomorphism Algebras of Abelian varieties with special reference to Superelliptic Jacobians}
\author[Yuri\ G.\ Zarhin]{Yuri\ G.\ Zarhin}
\thanks{Partially supported by  Simons Foundation  Collaboration grant  \# 585711. \newline
Part of this work was done in  May-June 2018 at the Max-Planck-Institut f\"ur Mathematik (Bonn, Germany), whose hospitality and support are gratefully acknowledged.}
\address{Department of Mathematics, Pennsylvania State University,
University Park, PA 16802, USA}

 \email{zarhin\char`\@math.psu.edu}

\begin{document}

\begin{abstract}
This is (mostly) a survey article.
 We use an information about Galois properties of points of small order on an abelian variety
 in order to describe its endomorphism algebra over an algebraic closure of the ground field.
 We discuss in detail applications to jacobians of cyclic covers of the projective line.

2000 Math. Subj. Class: Primary 14H40; Secondary 14K05, 11G30,
11G10

\end{abstract}

\maketitle

\section{Definitions and statements}
\label{one}
\label{endo} Throughout this paper $K$ is a field and $K_a$  its
algebraic closure. We write $K^{\sep}\subset K_a$ for the separable algebraic closure of $K$ in $K_a$
and
 $\Gal(K)$ for the absolute Galois
group $\Gal(K^{\sep}/K)=\Aut(K_a/K)$.
Throughout the paper  $\ell$ is a prime different from
$\fchar(K)$.
 If $A$ is a finite set then we write $|A|$ for its cardinality.
For every  abelian varieties $X$  and
$Y$ over $K_a$  we write $\Hom(X,Y)$ for the group of all
$K_a$-homomorphisms from $X$ to $Y$.

 If $X$ is an abelian variety of positive dimension
 over $K$ then $\End_K(X)$ and  $\End(X)$ stand for the rings of all its $K$-endomorphisms and
$K_a$-endomorphisms respectively. It  is known \cite{LangAV} that all endomorphisms of $X$ are defined over $K^{\sep}$.

The ring  $\End_K(X)$ is a subring of  $\End(X)$  and they both have the same identity element (automorphism), which we denote by $1_X$. We write
$\End_K^0(X)$
and $\End^0(X)$ for the corresponding
$\Q$-algebras $\End_K(X)\otimes\Q$ and
 $\End(X)\otimes\Q$; they both are semisimple finite-dimensional algebras over the field $\Q$ of rational numbers. We have
 $$\Q\cdot 1_X \subset \End_K^0(X) \subset \End^0(X).$$

The aim of this paper is to explain how one may obtain some information about the structure of $\End^0(X)$ in certain favorable circumstances,
 knowing only
the Galois properties of certain points of prime order and the ``multiplicities'' of the action of a certain endomorphism field on the differentials of the
 first kind on $X$.  One may view this paper as an exposition of ideas
that were developed  in \cite{ZarhinL} and \cite{ZarhinMZ2,ZarhinMZ2arxiv} 
and applied to {\sl superelliptic}  jacobians and prymians \cite{ZarhinBSMF,ZarhinL,ZarhinCamb,ZarhinIsk,ZarhinL2}. 
We also use this opportunity to correct 
 inaccuracies in the statements of Theorems 1.1(ii), 3.12(ii), 5.2(ii) and Remark 3.2 of  \cite{ZarhinMZ2} and fill  gaps in the proof of Theorem 3.12(ii) \cite[p. 702]{ZarhinMZ2}   in \cite[p. 697]{ZarhinMZ2}). (See also \cite{ZarhinMZ2arxiv} for the corrected version of  \cite{ZarhinMZ2}.) We also fill a gap in   the proof of \cite[Theorem 4.2,(i) and (ii)(a)]{ZarhinL} (caused by improper use of   \cite[Theorem 4.3.2]{Herstein} in \cite[Remark 4.1]{ZarhinL}), see below Theorems \ref{inequality} and \ref{inequality2} and their proofs (Section \ref{abelianMult}). 

Here is a couple of sample results that deal with jacobians $J(C_{f,p})$ of (smooth projective models of) superelliptic curves
$$C_{f,p}: y^p=f(x).$$  
Hereafter $p$ is a prime and we assume that $\fchar(K)\ne p$ while  $f(x)\in K[x]$ is a separable polymomial of degree $n \ge 3$.
 We write $\Z[\zeta_p]$  for the ring of integers in the $p$th cyclotomic field $\Q(\zeta_p)$.
(When $p=2$ we have $\Z[\zeta_p]=\Z$ and $C_{f,2}$ becomes the hyperelliptic curve $y^2=f(x)$.) 
The choice of a  primitive
$p$th root of unity in $K_a$ gives rise to a
 natural ring  embedding  
$$\Z[\zeta_p]\hookrightarrow \End(J(C_{f,p}))$$
(see \cite{SPoonen,Poonen} and Section \ref{superE} below).  If $p$ does {\sl not} divide $n$ then the dimension
of $J(C_{f,p})$ is $(n-1)(p-1)/2$; otherwise it is  $(n-2)(p-1)/2$.

\begin{thm}[see Th. 2.1 of \cite{ZarhinMRL}, Th. 2.1 of \cite{ZarhinBSMF} and Th. 3.8 of \cite{ZarhinL}]
\label{sampleHyper}
Let us assume that  $\fchar(K) \ne 2$ and $f(x)\in K[x]$ is an irreducible polynomial of degree $n \ge 5$,
whose Galois group $\Gal(f)$ over $K$ enjoys one of the following two properties.
\begin{itemize}
\item
 $\fchar(K) \ne 3$ and $\Gal(f)$ is either the full symmetric group $\ST_n$ or the alternating group $\mathbf{A}_n$;
\item
$n \in \{11,12,22,23,24\}$ and $\Gal(f)$ is isomorphic to the corresponding  Mathieu group $\mathbf{M}_n$.
\end{itemize}
Let $C_{f,2}: y^2=f(x)$ be the corresponding hyperelliptic curve of genus $[(n-1)/2]$ over $K$ and $J(C_{f,2})$ its jacobian, which is a $[(n-1)/2]$-dimensional abelian variety over $K$.

Then $\End(J(C_{f,2}))=\Z$. In particular, $J(C_{f,2})$ is absolutely simple.
\end{thm}

\begin{thm}[see Th. 1.1 of \cite{ZarhinCamb}]
\label{sampleSuper}
Let us assume that  $\fchar(K)=0$ and $f(x)\in K[x]$ is an irreducible polynomial of degree $n \ge 5$,
whose Galois group $\Gal(f)$ over $K$ is either the full symmetric group $\ST_n$ or the alternating group $\mathbf{A}_n$. Let $p$ be an odd prime,  
 $C_{f,p}$ the corresponding superelliptic curve over $K$
 and $J(C_{f,p})$ its jacobian, which is an abelian variety over $K$.
 
Then $\End(J(C_{f,p})) =\Z[\zeta_p]$. In particular, $J(C_{f,p})$ is absolutely simple.
\end{thm}

\begin{thm}[see Th. 1.1 of \cite{ZarhinMZ2}, Th. 1.1 of \cite{ZarhinMZ2arxiv} and Theorem \ref{mainJfq}   below]\footnote{In Th. 1.1 of \cite{ZarhinMZ2} the assertion (ii)(a) actually is not proven and should be ignored.}
\label{doublyTransitive}
Suppose that $K$ has characteristic zero, $n\ge 4$ and $p$ is an odd prime
that does not divide $n$. Assume also that either $n = p + 1$ or $p$ does not divide
$n -1$.

Suppose that $K$ contains a primitive $p$th root of unity and
$\Gal(f)$ is a doubly transitive permuation group (on the set of roots of $f(x)$) that  does not contain a proper normal subgroup, whose
index divides $n-1$.


 Then
 $\End(J(C_{f,p}))=\Z[\zeta_p]$. In particular, $J(C_{f,p})$ is absolutely simple.
 \end{thm}

The paper is organized as follows. Section \ref{sec1} contains basic definitions and reviews  elementary results concerning
 the structure of $\End^0(X)$ and $\End_K^0(X)$ under certain assumptions on the Galois properties of the group $X_{\ell}$ of points of prime
order $\ell$ on $X$ related to the image $\tilde{G}_{\ell,X,K}$ of the Galois group in $\Aut(X_{\ell})$. These results are generalized in
Section \ref{sec2} when $X$ admits multiplications from the ring $\OC$ of integers in a number field $E$ and $X_{\ell}$ is replaced by the
group $X_{\lambda}$ of points on $X$ that are killed by multiplication from a maximal ideal $\lambda\subset \OC$. (The results of Section
\ref{sec1} correspond to the case $\OC=\Z, E=\Q, \lambda=\ell\Z$.) In order to prove the results of Section \ref{sec2}, we need to use results
 from the theory of (central semi)simple algebras over fields, which are discussed in Section \ref{sec3}. We  prove the assertions of
Section \ref{sec2}
in Section \ref{abelianMult}. In Section \ref{tangentS} the Lie algebra $\Lie(X)$  of $X$ (which is the dual of the space of differentials of the first kind)
 enters the picture: assuming that $\fchar(K)=0$, we discuss the action of $E$ on  $\Lie(X)$, which allows us to extend the results of Section \ref{sec2}.
We are going to apply these results to superelliptic (hypergeometric) jacobians $J(C_{f,q})$ of curves $C_{f,q}$ and their natural abelian subvarieties $J^{(f,q)}$, which are provided with
the action of the $q$th cyclotomic field $E=\Q(\zeta_q)$ where $q$ is a prime power. (Here $C_{f,q}$ is the smooth projective module of the affine curve
 $y^q=f(x)$ where $f(x)$ is a polynomial without multiple roots.)
 In order to do this, we need to discuss certain constructions
related to permutation groups and permutation modules, which is done in Section \ref{permutation}. Section \ref{superE} contains results
about endomorphism algebras of  $J^{(f,q)}$. Section \ref{heartqn} contains auxiliary results about the structure of the Galois module $J^{(f,q)}_{\lambda}$
where $\lambda$ is the maximal ideal of the $q$th cyclotomic ring $\Z[\zeta_q]$ generated by $(1-\zeta_q)$.

{\bf Acknowledgements}. I am deeply grateful to Jiangwei Xue, who had read the first version of the manuscript and made numerous valuable comments and suggestions that helped to improve the exposition. 
Part of this  work was done in June 2017 during my stay at Steklov Mathematical Institute (Russian Academy of Sciences, Moscow), whose hospitality is gratefully acknowledged.

\section{Definitions and first statements}
\label{sec1}

 \begin{sect}
 \label{split0}
 We write $C_{K,X}$ and $C_X$ for the centers of  $\End_K^0(X)$ and $\End^0(X)$. Both  $C_{K,X}$ and $C_X$  are isomorphic to direct sums of number fields; each of those fields is either totally real or CM. It is well known that $X$ is $K$-isogenous to a self-product of a $K$-simple  abelian variety $Z_K$  (respectively, is isogenous over $K_a$ to a self-product of an absolutely simple abelian variety $Z$ over $K_a$) if and only if $C_{K,X}$ (respectfully, $C_X$) is a  field. If this is the case then
  there is a canonical isomorphism between the fields $C_{K,X}$ and $C_{K,Z_K}$ (respectfully between  the fields $C_X$ and $C_Z$).  In addition,  $C_X$  is a field if and only if $\End^0(X)$  is a simple $\Q$-algebra. In general,
the semisimple $\Q$-algebra $\End^0(X)$ splits into a finite direct
sum
$$\End^0(X)= \sum_{s\in \I(X)} D_s$$
 of simple $\Q$-algebras $D_s$. (Here the finite nonempty set  $\I(X)$ is identified with the
 set of (nonzero) minimal two-sided ideals in $\End^0(X)$.) Let $e_s$ be the identity element of $D_s\subset \End^0(X)$. We have
 $$1_X=\sum_{s\in \I(X)}e_s\in \End^0(X), \ e_s^2=e_s, \ e_s e_t=0 \ \forall s \ne t.$$
 Let us choose a positive integer $N$ such that all $Ne_s \in \End(X)$ and consider
 $$X_s:=(Ne_s)(X) \subset X,$$
 which is an abelian subvariety of $X$ that is defined over $K_a$.

 The following assertion is contained in \cite[Remark 1.4 on pp. 192-193]{ZarhinL}.

 \begin{lem}
 \label{remL}
 \begin{itemize}
 \item[(i)]
 The $\Q$-algebras $D_s$ and $\End^0(X_s)$ are isomorphic. In particular, $\End^0(X_s)$
 is a simple $\Q$-algebra, i.e., $X_s$ is isogenous over $K_a$ to a self-product of simple abelian variety over $K_a$.
 \item[(ii)]
 $\Hom(X_s,X_t)=\{0\}$ for each $s \ne t$ .
 \item[(iii)]
 The natural $K_a$-homomorphism of abelian varieties
 $$\Pi_X: \prod_{s\in \I(X)}X_s \to X, \  \{x_s\}_{s\in \I(X)} \mapsto \sum_{s\in \I(X)} x_s$$
 is an isogeny.
 \end{itemize}
 \end{lem}
 \end{sect}

 \begin{sect}
 \label{kappaX}
 Since $X$ is defined over $K$, each $\sigma \in \Gal(K)$ and $u \in \End(X)$ give rise to $^{\sigma}u \in \End(X)$ such that
 $$^{\sigma}u(x)=\sigma (u (\sigma^{-1}x)) \ \forall x \in X(K_a).$$
 This gives us a continuous group homomorphism \cite{Silverberg}
 $$\kappa_{X,K}: \Gal(K) \to \Aut(\End(X)), \ k_X(\sigma)(u)= ^{\sigma}u \ \forall \sigma \in \Gal(K), u \in \End(X)$$
 with finite image.
 (Here $\Aut(\End(X))$ is provided with discrete topology).
 If $L/K$ is a finite separable algebraic field extension with $L\subset K^{\sep}$ then $\Gal(L)$ is an open  subgroup of finite index in $\Gal(K)$ and
 the restriction of $\kappa_{X,K}$ to $\Gal(L)$ coincides with
 $$\kappa_{X,L}: \Gal(L) \to \Aut(\End(X)).$$
 It is well known that
 $\End_L(X)$ coincides  with the subring $\End(X)^{\Gal(L)}$ of $\Gal(L)$-invariants, i.e.,
 $$\End_L(X)=\{u\in\End(X)\mid ^{\sigma}u =u \ \forall \sigma \in \Gal(L)\}.$$
 In particular,
$$\End_K(X)=\End(X)^{\Gal(K)}=\{u\in\End(X)\mid ^{\sigma}u =u \ \forall \sigma \in \Gal(K)\}.$$
The kernel $\ker(\kappa_{X,K})$ is a closed normal  subgroup of finite index in $\Gal(K)$ and therefore is open, i.e.
coincides with the Galois (sub)group $\Gal(\mathcal{F}_{X,K})$ of a certain overfield $\mathcal{F}_{X,K}\supset K$ such that $\mathcal{F}_{X,K}\subset K^{\sep}$ and
$\mathcal{F}_{X,K}/K$ is a finite Galois extension. Clearly, $\End_L(X)=\End(X)$ (i.e., all endomorphisms of $X$ are defined over $L$) if and only if $L\supset \mathcal{F}_{X,K}$.
In general, $\mathcal{F}_{X,L}$ coincides with the compositum $\mathcal{F}_{X,K} L$ of $\mathcal{F}_{X,K}$ and $L$ in $K^{\sep}$.

The following assertion is contained in \cite[Remark 1.4 on pp. 192-193]{ZarhinL}.

\begin{lem}
\label{transL}
The finite subset $\{Ne_s\mid s \in \I(X)\}$ of $\End(X)$ is $\Gal(K)$-stable.
If $\End_K(X)$ has no zero divisors then the action of $\Gal(K)$ on $\I(X)$ is transitive
and
$$\dim(X_s)=\dim(X)/|\I(X)|,$$
which does not depend on a choice of $s\in \I(X)$.
\end{lem}

 \begin{cor}
 \label{divL}
 If $\End_K^0(X)$ is a number field then the action of $\Gal(K)$ on $\I(X)$ is transitive
 and
 $|\I(X)|$ divides $\dim(X)$.
 \end{cor}

 \begin{proof}
 Since  $\End_K^0(X)$ is a number field, $\End_K(X)$ is an order in this field and therefore has no zero divisors.
 So, we may apply Lemma \ref{transL} and get the desired  transitivity  and the equality $\dim(X_s)=\dim(X)/|\I(X)|$. Since all three numbers
 $\dim(X_s),\dim(X)$ and $|\I(X)|$ are nonzero integers, we conclude that  $|\I(X)|$ divides $\dim(X)$.
 \end{proof}


 \begin{thm}
 \label{transitive0}
 Let $F/K$ be a finite Galois field extension such that $F\subset K^{\sep}$ and all endomorphisms of $X$ are defined over $F$.
 If  $\End_K^0(X)$ is a number field and $\Gal(F/K)$ does not contain a proper subgroup, whose index divides $\dim(X)$  then
 $\I(X)$  is a singleton, i.e., $\End^0(X)$ is a simple $\Q$-algebra.
 \end{thm}

 \begin{proof}
 Since all endomorphisms of $X$ are defined over $F$,
 $$F\supset \mathcal{F}_{X,K},  \ \Gal(F)\subset \Gal(\mathcal{F}_{X,K})$$
 and $\kappa_{X,K}: \Gal(K) \to \Aut(\End(X))$ factors through the quotient $\Gal(K)/\Gal(F)=\Gal(F/K)$.
 This implies that the action of $\Gal(K)$ on $\I(X)$ also factors through $\Gal(F/K)$.
 By Corollary \ref{divL} $\Gal(K)$ acts transitively on $\I(X)$ and therefore the corresponding
 $\Gal(F/K)$-action on $\I(X)$ is also transitive. This implies that $\Gal(F/K)$ has a subgroup of index $|\I(X)|$.
 By Corollary \ref{divL}, $|\I(X)|$ divides $\dim(X)$ and therefore this subgroup must coincide with the whole
 $\Gal(F/K)$, i.e.,  $\I(X)$  is a singleton.
 \end{proof}
 \end{sect}

 Let $X_{\ell}$ be the kernel of multiplication by $\ell$ in $X(K_a)$. It is well known \cite{LangAV,MumfordAV}  that $X_{\ell}$  is a $\Gal(K)$-invariant subgroup of $X(K^{\sep})$, which is (as a group) a $2\dim(X)$-dimensional
 vector space over the prime finite field $\F_{\ell}$ of characteristic $\ell$. This gives rise to the natural continuous group homomorphism
 $$\tilde{\rho}_{\ell,X,K}: \Gal(K) \to \Aut_{\F_{\ell}}(X_{\ell}),$$
 whose image we denote by $\tilde{G}_{\ell,X,K}$. By definition, we get the surjective continuous  homomorphism
 $$\tilde{\rho}_{\ell,X,K}: \Gal(K) \twoheadrightarrow  \tilde{G}_{\ell,X,K}\subset \Aut_{\F_{\ell}}(X_{\ell}).$$
 One may view the vector space  $X_{\ell}$ as  (faithful) $\tilde{G}_{\ell,X,K}$-module.

 The next well known lemma goes back to K. Ribet \cite{Ribet3} and S. Mori \cite{Mori}.

 \begin{lem}
 \label{moriZ}
 ( \cite[Lemma 1.2 on p. 191]{ZarhinL})
 If the centralizer
 $$\End_{\tilde{G}_{\ell,X,K}}(X_{\ell})=\F_{\ell}$$
 then
 $$\End_K(X)=\Z, \ \End_K^{0}(X)=\Q.$$
 \end{lem}

 The next statement follows readily  from  \cite[Th. 1.5 on pp. 193--194]{ZarhinL}.

 \begin{thm}
 \label{endoZ}
 Let us assume that
  $\End_{\tilde{G}_{\ell,X,K}}(X_{\ell})$ is a field.
 Suppose that $\tilde{G}_{\ell,X,K}$   does not contain a proper subgroup, whose  index divides $\dim(X)$.
 Then $\End^0(X)$ is a  simple $\Q$-algebra.
 \end{thm}

 The following assertion is an immediate corollary of Theorem \ref{endoZ} and \cite[Th. 1.6 on pp. 195]{ZarhinL}.

  \begin{thm}
 \label{endoZ2}
 Let us assume that
  $$\End_{\tilde{G}_{\ell,X,K}}(X_{\ell})=\F_{\ell}.$$
 Suppose that $\tilde{G}_{\ell,X,K}$   does  contain neither a proper subgroup with  index dividing $\dim(X)$
 nor a normal subgroup of index 2.
 Then $\End^0(X)$ is a  central simple $\Q$-algebra.
 \end{thm}

 \section{Abelian varieties with multiplication}
\label{sec2}
 In this section we discuss analogues of results of Section \ref{one} when the endomorphism algebra of an abelian variety contains a given number field.

 \begin{sect}
Let $E$ be a number field and
$$i: E \hookrightarrow  \End_K^0(X) \subset \End^0(X)$$
be a $\Q$-algebra embedding
such that $i(1)=1_X$.  It is known \cite[Prop. 2 on p.
36]{Shimura2}) that the degree $[E:\Q]$ divides $2\dim(X)$. Let us put
$$d_{X,E}=\frac{2\dim(X)}{[E:\Q]}.$$
We write $\End^0(X,i)$ for the centralizer of $i(E)$ in $\End^0(X)$ and $\End_K^0(X,i)$ for the centralizer of $i(E)$ in $\End_K^0(X)$.  We have
$$i(E)\subset \End_K^0(X,i)\subset \End^0(X,i)\subset \End^0(X), \  \End_K^0(X,i)\subset \End_K^0(X)\subset  \End^0(X).$$
 We write $i(E)C_X$ for the compositum of $i(E)$ and $C_X$ in  $\End^0(X)$. In other words, $i(E)C_X$ is the image of the homomorphism of $\Q$-algebras
$$i\otimes \mathrm{id}_{C_X}: E\otimes_{\Q} C_X \to  \End^0(X),  \  e\otimes c \mapsto i(e)c.$$
Clearly $E\otimes_{\Q} C_X$ is a direct sum of fields, each of which contains a subfield isomorphic to $E$.
This implies that $i(E)C_X$  is a direct sum of fields, each of which contains a subfield isomorphic to $E$.
(In addition, each such a field contains a subfield isomorphic to $C_X$ if the latter is a field.)

 Clearly,  $i(E)C_X$ commutes with $i(E)$ and therefore lies in  $\End^0(X,i)$ and even in its center.

The next three assertions will be proven in in Section \ref{abelianMult}.

The first one
 is a corollary of standard facts about centralizers and bicentralizers of semisimple subalgebras of semisimple algebras. (See Theorem \ref{ssCenter} below.)

\begin{thm}
\label{SCenter}
$\End^0(X,i)$ is a finite-dimensional semisimple $\Q$-algebra, whose center coincides with $i(E)C_X$.
\end{thm}

The next two statements deal with the $E$-dimension of $\End^0(X,i)$.

\begin{thm}
\label{boundDim}
Let us consider $\End^0(X,i)$ as an $E$-algebra. Then the $E$-algebra $\End^0(X,i)$ is semisimple and
$$\dim_E(\End^0(X,i)) \le\left(\frac{2\dim(X)}{[E:\Q]}\right)^2.$$
\end{thm}

\begin{thm}
\label{bigDim}
Suppose that
$$\dim_E(\End^0(X,i)) =\left(\frac{2\dim(X)}{[E:\Q]}\right)^2.$$
Then $E$ contains $C_X$ and therefore $C_X$ is a field. In addition,  $\End^0(X,i)$  is a central simple $E$-algebra and $X$ is an abelian variety of CM type over $K_a$. In particular, $X$ is isogenous over $K_a$ to a self-product of an absolutely simple abelian variety of CM type over $K_a$.
\end{thm}

\begin{ex}
Let $E=\Q$. Then $\End^0(X,i)=\End^0(X)$. We have
$$\dim_{\Q}(\End^0(X))\le (2g)^2;$$
 the equalty holds if and only if $\fchar(K_a)>0$ and $X$ is isogenous over $K_a$ to a self-product of a supersingular elliptic curve  \cite{ZarhinMRL}.
\end{ex}

\end{sect}

\begin{sect}
\label{OCE}
Let $\OC$ be the ring of integers in $E$.  If $\lambda$ is a maximal ideal in $\OC$ then we write $k(\lambda)$ for its (finite) residue field $\OC/\lambda$. For all but finitely many $\lambda$
 $$\fchar(k(\lambda))\ne \fchar(K).$$
Let us assume that
$$i(\OC)\subset \End_K(X).$$
Then the center of $\End_K(X,i)$ contains $i(\OC)$ and  $\End_K(X,i)$ becomes an $i(\OC)\cong \OC$-algebra. Notice that $\OC$ is a Dedekind ring and the $\OC$-module $\End_K(X)$ is finitely generated torsion-free. Therefore
$\End_K(X)$
is isomorphic (as an $\OC$-module) to a direct sum  of finitely many nonzero  ideals of  $\OC$.
Let us assume that $\fchar(k(\lambda))\ne \fchar(K)$ and consider
$$X_{\lambda}=\{x \in X(K_a)\mid i(u) x=0 \ \forall u \in \lambda \subset \OC\} \subset X(K_a).$$
It is known \cite{Ribet2} that $X_{\lambda}$ is a $\Gal(K)$-invariant finite subgroup of $X(K^{\sep})$ that carries the natural structure of $d_{X,E}$-dimensional vector space over $k(\lambda)$. The Galois action on
$X_{\lambda}$ induces the continuous group homomorphism
$$\bar{\rho}_{\lambda,X,K}: \Gal(K) \to \Aut_{k(\lambda)}(X_{\lambda}), $$
whose image we denote by $\tilde{G}_{\lambda,X,K}$. As above (in the case of $E=\Q, \OC=\Z, \lambda=\ell\Z)$), we get the surjective continuous group homomorphism
$$\bar{\rho}_{\lambda,X}=\bar{\rho}_{\lambda,X,K}: \Gal(K)\twoheadrightarrow \tilde{G}_{\lambda,X,K} \subset \Aut_{k(\lambda)}(X_{\lambda}).$$
If $K^{\prime} \subset K^{\sep}$ is an overfield of $K$ then $\bar{\rho}_{\lambda,X,K^{\prime}}$ coincides with the restriction of $\bar{\rho}_{\lambda,X,K}$ to $\Gal(K^{\prime})\subset \Gal(K)$.

Let $K(X_{\lambda})\subset K^{\sep}$ be the field of definition of all points of $X_{\lambda}$. Then the subgroup $\Gal(K(X_{\lambda}))$ of $\Gal(K)$ coincides with
$\ker(\bar{\rho}_{\lambda,X,K})$,  $K(X_{\lambda})/K$ is a finite Galois extension and  $\bar{\rho}_{\lambda,X,K}$ induces the canonical isomorphism
$$\Gal(K(X_{\lambda})/K)=\Gal(K)/\Gal(K(X_{\lambda})) \cong \tilde{G}_{\lambda,X,K}\subset \Aut_{k(\lambda)}(X_{\lambda}).$$
\end{sect}

\begin{sect}
\label{minS}
We will need the following result related to the notion of  minimal covers of groups
 \cite{FT}.

\begin{lem}
\label{maximal}
Let $F/K$ be a finite Galois field extension  and let $L/K$ be a Galois field extension such that
$$K\subset L \subset F.$$
Then there exists an overfield $\mathcal{K}$ of $K$  that is a subfield of $F$ and enjoys the following properties.
\begin{itemize}
\item[(i)]
$K\subset \mathcal{K} \subset F$.
\item[(ii)]
Let $\phi_{\mathcal{K},L}$ be the restriction of the natural surjective group homomorphism $\Gal(F/K) \twoheadrightarrow \Gal(L/K)$
to $\Gal(F/\mathcal{K}) \subset \Gal(F/K)$. Then the group homomorphism
$\phi_{\mathcal{K},L}: \Gal(F/\mathcal{K}) \to \Gal(L/K)$
is surjective.
\item[(iii)]
$\mathcal{K}$ is maximal among the fields that satisfy (i) and (ii).
\end{itemize}
\end{lem}

\begin{proof}
Clearly, $\mathcal{K}=K$ satisfies (i) and (ii). The existence of maximal $\mathcal{K}$ follows from the finiteness of the set of intermediate fields that satisfy (i).
\end{proof}

\begin{rem}
\label{minR}
\begin{itemize}
\item[(i)]
The maximality of $\mathcal{K}$  in Lemma \ref{maximal} means that  surjective $\phi_{\mathcal{K},L}: \Gal(F/\mathcal{K}) \to \Gal(L/K)$ is a {\sl minimal cover} in a sense of \cite{FT},
i.e., if $H$ is a subgroup of $\Gal(F/\mathcal{K}) $ that maps {\sl onto}  $\Gal(L/K)$ then $H=\Gal(F/\mathcal{K})$. Indeed, the  subfield  $F^{H}$ of $F$ enjoys the properties (i--ii) and  contains
$F^{\Gal(F/\mathcal{K})}=\mathcal{K}$. In light of the maximality of $\mathcal{K}$, we have $F^{H}=\mathcal{K}$ and therefore  $\Gal(F/\mathcal{K}) =H$. (Such a $\mathcal{K}$ is not necessarily unique.)
\item[ii)]
Suppose that $H$ is a subgroup in $\Gal(F/\mathcal{K}) $ of index $d>1$. By (i),  the index $d^{\prime}:=(\Gal(L/K):\phi_{\mathcal{K},L}(H))>1.$  I claim that
$d^{\prime}$ divides $d$. Indeed, if $\phi=\phi_{\mathcal{K},L}$ then
$$d=\frac{|\Gal(F/\mathcal{K})|}{|H|}=\frac{|\ker(\phi)|\cdot |\Gal(L/K)|}{|\ker(\phi)\bigcap H| |\phi(H)|}=$$
$$\frac{|\ker(\phi)|}{|\ker(\phi)\bigcap H|} \cdot \frac{|\Gal(L/K)|}{ |\phi(H)|}=
\frac{|\ker(\phi)|}{|\ker(\phi)\bigcap H|} \cdot d^{\prime}.$$
Since $\ker(\phi)\bigcap H$ is a subgroup of $\ker(\phi)$, Lagrange's theorem tells us that $|\ker(\phi)\bigcap H|$ divides $|\ker(\phi)|$ and therefore $d^{\prime}$ divides $d$.

This implies that if $d>1$ is an integer such that $\Gal(L/K)$ does {\sl not} contain a proper subgroup of index dividing $d$  then
$ \Gal(F/\mathcal{K}) $ also does {\sl not} contain a proper subgroup of index dividing $d$.
\end{itemize}
\end{rem}

\begin{rem}
\label{disjoint}
Let $K,L,F$ be as in Lemma \ref{maximal}.
Suppose that $\mathcal{T}$ is a field that is an overfield of $K$ and a subfield of $F$.
Since the field extension  $L/K$ is Galois, the field extension $\mathcal{T}L/\mathcal{T}$ is also Galois. Hereafter $\mathcal{T}L$ is   the compositum of
 $\mathcal{T}$ and $L$, which is a subfield of $F$ with
 \begin{equation}
 \label{compDEG}
  [\mathcal{T}L:K] \le [\mathcal{T}:K][L:K];
  \end{equation}
  the equality holds if and only if  $\mathcal{T}$ and $L$ are {\bf  linearly disjoint} over $K$.

The assertion that  $\mathcal{T}$  enjoys the property (ii) of Lemma \ref{maximal} means that  $\mathcal{T}$ and $L$ are {\bf  linearly disjoint} over $K$.
Indeed, suppose that  $\mathcal{T}$ and $L$ are {\bf  linearly disjoint} over $K$. Then
$$[\mathcal{T}L:K]=[\mathcal{T}:K][L:K].$$
 Since
$$[\mathcal{T}L:K]=[\mathcal{T}L:\mathcal{T}] [\mathcal{T}:K],$$
we conclude that $[\mathcal{T}L:\mathcal{T}]=[L:K]$ and therefore the natural injective group homomorphism (``restriction'' to $L$)
$$\mathrm{res}_L: \Gal(\mathcal{T}L/\mathcal{T}) \to \Gal(L/K)$$ is a map between two finite groups of the same order $[L:K]$ and therefore
is an isomorphism.  Notice that $\mathrm{res}_L$ coincides with the restriction to  $\Gal(\mathcal{T}L/\mathcal{T})\subset \Gal(F/K)$ of
$\phi_{\mathcal{T},L}: \Gal(F/\mathcal{T}) \to \Gal(L/K)$. This implies that $\phi_{\mathcal{T},L}$ is surjective, i.e., $\mathcal{T}$  enjoys the property (ii) of Lemma \ref{maximal}.

Conversely, let us assume that $\phi_{\mathcal{T},L}$ is {\sl surjective}. Notice that $\phi_{\mathcal{T},L}$ factors through
$\Gal(F/\mathcal{T}) \twoheadrightarrow \Gal(\mathcal{T}L/\mathcal{T})$ and therefore the surjectiveness of $\phi_{\mathcal{T},L}$  implies (actually, is equivalent to)
the surjectiveness of
$$\mathrm{res}_L: \Gal(\mathcal{T}L/\mathcal{T})\to \Gal(L/K),$$
which, in turn, implies the inequality $[\mathcal{T}L:\mathcal{T}]\ge [L:K]$. This implies that
$$[\mathcal{T}L:K]=[\mathcal{T}L:[\mathcal{T}] [\mathcal{T}:K] \ge [L:K]  [\mathcal{T}:K] ,$$
which tells us in light of \eqref{compDEG} that
$$[\mathcal{T}L:K]= [L:K]  [\mathcal{T}:K] ,$$
i.e., $\mathcal{T}$ and $L$ are   linearly disjoint over $K$.

This means that  $\mathcal{T}$ enjoys the properties (i)-(iii)  of Lemma \ref{maximal} if and only if it is  {\sl maximal} among overfields of $K$ that lie in $F$ and are linearly disjoint with $L$ over $K$.
\end{rem}

\begin{rem}
\label{minLambda}
Let us apply Lemma \ref{maximal} and Remark \ref{minR} to $L=K(X_{\lambda})$ and choose as $F\subset K^{\sep}$ any finite Galois extension of $K$ that contains both $K(X_{\lambda})$ and $\mathcal{F}_{X,K}$;
in particular,  all endomorphisms of $X$ are defined over $F$.  We have
$$\Gal(L/K)=\Gal(K(X_{\lambda})/K)=\tilde{G}_{\lambda,X,K}.$$
Clearly,
$\bar{\rho}_{\lambda,X,K}$ factors through
$\Gal(K)/\Gal(F)=\Gal(F/K)$, and for each overfield $K^{\prime}\subset F$ of $K$ the image
$$\tilde{G}_{\lambda,X,K^{\prime}}=\bar{\rho}_{\lambda,X,K}(\Gal(K^{\prime}))$$
coincides with the image of
$$\Gal(F/K^{\prime}) \to \Gal(K(X_{\lambda})/K^{\prime})=\tilde{G}_{\lambda,X,K^{\prime}}\subset\tilde{G}_{\lambda,X,K}\subset \Aut_{k(\lambda)}(X_{\lambda}).$$
Now if we take as $K^{\prime}$  a field  $\mathcal{K}$  that enjoys the properties (i)-(iii) of Lemma \ref{maximal} then
$$\tilde{G}_{\lambda,X,\mathcal{K}}=\tilde{G}_{\lambda,X,K}\subset \Aut_{k(\lambda)}(X_{\lambda})$$
 and the surjective group homomorphism
$$\phi_{\mathcal{K}}:\Gal(F/\mathcal{K}) \to \Gal(L/K)=\tilde{G}_{\lambda,X,K}$$
is a {\sl minimal cover}. In particular,
 $$\End_{\tilde{G}_{\lambda,X,K}}(X_{\lambda})=\End_{\tilde{G}_{\lambda,X,\mathcal{K}}}(X_{\lambda}).$$
  In addition, if $d>1$ is a positive integer such that $\tilde{G}_{\lambda,X,K}$ does {\sl not} contain a proper subgroup, whose index divides $d$
  then $\Gal(F/\mathcal{K})$ also does {\sl not} contain a proper subgroup,  whose index divides $d$. Notice also that since all the endomorphisms of $X$ are defined over $F$, i,e.,
   $\kappa_{X,K}$  kills $\Gal(F)$,
  there is the natural homomorphism
  $$\Gal(F/K)=\Gal(K)/\Gal(F) \to \Aut(\End(X,i))$$
  induced by $\kappa_{X,K}$  such that
  $$\End_{K^{\prime}}(X,i)=\End(X,i)^{\Gal(F/K^{\prime})}$$
  for all fields $K^{\prime}$ with $K\subset K^{\prime}\subset F$, including $K^{\prime}=\mathcal{K}$ or $K$.
\end{rem}

\end{sect}


\begin{lem}
 \label{moriO}
 (\cite[Lemma 3.8 on p. 700]{ZarhinMZ2}])
 If the centralizer
 $$\End_{\tilde{G}_{\lambda,X,K}}(X_{\lambda})=k(\lambda)$$
 then
 $\End_K(X,i)=i(\OC).$
 \end{lem}

 Since the natural $\Q$-algebra homomorphisms
 $$\OC\otimes \Q \to E, \ i(\OC)\otimes\Q \to i(E)$$
 are obvious isomorphisms, Lemma \ref{moriO} implies the following assertion.

 \begin{cor}
 \label{moriE}
 If the centralizer
 $$\End_{\tilde{G}_{\lambda,X,K}}(X_{\lambda})=k(\lambda)$$
 then
 $\End_K^{0}(X,i)=i(E).$
 \end{cor}

 \begin{thm}
 \label{endoE}
 Let us assume that
   $$\End_{\tilde{G}_{\lambda,X,K}}(X_{\lambda})=k(\lambda).$$
 Suppose that $\tilde{G}_{\lambda,X,K}$   does not contain a proper subgroup, whose  index divides $d_{X,E}$.
 Then:
 \begin{itemize}
 \item[(i)]
  $\End^0(X)$ is a simple $\Q$-algebra;
  \item[(ii)]
   $i(E)$ contains $C_X$, i.e., the center $i(E)C_X$ of $\End^0(X,i)$ coincides with $i(E)$;
   \item[(iii)]
 $\End^0(X,i)$ is a central simple $i(E)$-algebra.
 \end{itemize}
 \end{thm}

 We prove Theorem   \ref{endoE}  in Section \ref{abelianMult}.

\section{Semisimple subalgebras of semisimple algebras}
\label{sec3}
This section contains auxiliary results about semisimple  algebras over fields that will be used in the proof of Theorems \ref{SCenter}, \ref{boundDim} and \ref{bigDim}
in Section \ref{abelianMult}.
\label{centralizerAL}
All associative algebras, subalgebras and rings are assumed to have $1$.
Let $k$ be a field, $\mathcal{A}$ a finite-dimensional central simple $k$-algebra. We write $\End(\mathcal{A})$  for the ring of endomorphisms of the additive abelian group $A$ and
 $\End_k(\mathcal{A})$ for the
$k$-algebra of endomorphisms of the $k$-vector space $\mathcal{A}$.  We have
$$k \cdot \mathrm{id}_{\mathcal{A}} \subset \End_k(\mathcal{A})\subset \End(\mathcal{A}) $$
where $\mathrm{id}_{\mathcal{A}}$ is the identity endomorphism of $\mathcal{A}$. One may view $\End_k(\mathcal{A})$ as the centralizer of $k \cdot \mathrm{id}_{\mathcal{A}} $ in $\End(\mathcal{A})$.
We write $\mathcal{A}^{\mathrm{opp}}$ for the opposite algebra of $\mathcal{A}$; it is well known that $\mathcal{A}^{\mathrm{opp}}$ is also simple central over $k$ and the natural $k$-algebra homomorphism
$$\mathcal{A}\otimes_k \mathcal{A}^{\mathrm{opp}} \to \End_k(\mathcal{A}), \ u\otimes v \mapsto \{x \mapsto uxv \ \forall \ x \in \mathcal{A}\}$$
is an isomorphism of (central simple $k$-algebras).  Further we will identify $\mathcal{A}\otimes_k \mathcal{A}^{\mathrm{opp}}$ with $\End_k(\mathcal{A})$ via this isomorphism and
$$\mathcal{A}=\mathcal{A}\otimes 1, \ \mathcal{A}^{\mathrm{opp}}=1\otimes \mathcal{A}^{\mathrm{opp}}$$
with corresponding $k$-subalgebras of $\End_k(\mathcal{A})$. It is well known that the centralizer of $\mathcal{A}\otimes 1$ (resp. of $1\otimes \mathcal{A}^{\mathrm{opp}}$) in
$\End(\mathcal{A})$  actually lies in $\End_k(\mathcal{A})$ (because both subalgebras contain $k\otimes 1=1\otimes k=k \cdot \mathrm{id}_{\mathcal{A}}$) and coincides with $1\otimes \mathcal{A}^{\mathrm{opp}}$ (resp. with $\mathcal{A}\otimes 1$).

Let  $\mathcal{B}$ be a   $k$-subalgebra of $\mathcal{A}$.
Let $\mathcal{Z}_{\mathcal{A}}(\mathcal{B})$ be the centralizer of $\mathcal{B}$ in $\mathcal{A}$. Clearly, $\mathcal{Z}_{\mathcal{A}}(\mathcal{B})$ is a $k$-subalgebra
of $\mathcal{A}$; in addition, $\mathcal{B}$ lies in the {\sl double centralizer} of $B$, i.e., in the centralizer $\mathcal{Z}_{\mathcal{A}}(\mathcal{Z}_{\mathcal{A}}(\mathcal{B}))$
 of $\mathcal{Z}_{\mathcal{A}}(\mathcal{B})$. It is also clear that the center of  $\mathcal{B}$ lies in the center
of $\mathcal{Z}_{\mathcal{A}}(\mathcal{B})$.
The following assertion is well known in the case of simple $\mathcal{B}$.

\begin{thm}
\label{ssCenter}
Suppose that $\mathcal{B}$ is a semisimple $k$-algebra. Then $\mathcal{Z}_{\mathcal{A}}(\mathcal{B})$  is also a semisimple $k$-algebra.
In addition, the centralizer of $\mathcal{Z}_{\mathcal{A}}(\mathcal{B})$  in $\mathcal{A}$ coincides with $\mathcal{B}$, i.e., $\mathcal{B}$ coincides with its own double centralizer in
 $\mathcal{A}$.

 In particular, the centers of $\mathcal{B}$  and $\mathcal{Z}_{\mathcal{B}}(\mathcal{A})$ do coincide.

 If, in addition, $\mathcal{B}$  is commutative then the center of
 $\mathcal{Z}_{\mathcal{A}}(\mathcal{B})$ coincides with  $\mathcal{B}$.
\end{thm}

 \begin{proof}
 The tensor product $\mathcal{B}\otimes_k \mathcal{A}^{\mathrm{opp}}$ is a {\sl semisimple} $k$-algebra, because $\mathcal{A}^{\mathrm{opp}}$ is central simple and $\mathcal{B}$ is simple.
 The algebra
 $$\mathcal{Z}_{\mathcal{A}}(\mathcal{B})=\mathcal{Z}_{\mathcal{A}}(\mathcal{B})\otimes 1 \subset A\otimes_k \mathcal{A}^{\mathrm{opp}}=\End_k(\mathcal{A})$$
coincides with the centralizer of the {\sl semisimple} algebra
$$\mathcal{B}\otimes_k \mathcal{A}^{\mathrm{opp}} \subset \mathcal{A}\otimes_k \mathcal{A}^{\mathrm{opp}}=\End_k(\mathcal{A}),$$
i.e.,  it is the endomorphism algebra of the {\sl semisimple} $\mathcal{B}\otimes_k \mathcal{A}^{\mathrm{opp}}$-module $\mathcal{A}$
and therefore is semisimple.
By the Jacobson density theorem, the double centralizer of
$$\mathcal{B}\otimes_k \mathcal{A}^{\mathrm{opp}}\subset \mathcal{A}\otimes_k \mathcal{A}^{\mathrm{opp}}=\End_k(\mathcal{A})$$
coincides with $\mathcal{B}\otimes_k \mathcal{A}^{\mathrm{opp}}$. On the other hand, if $\mathcal{C}$ is the double centralizer of $\mathcal{B}$ in $\mathcal{A}$ then
$\mathcal{C}$ contains $\mathcal{B}$ and $\mathcal{C}\otimes_k \mathcal{A}^{\mathrm{opp}}$ lies in the double centralizer of $\mathcal{B}\otimes_k \mathcal{A}^{\mathrm{opp}}$ , i.e.,
$$\mathcal{C}\otimes_k \mathcal{A}^{\mathrm{opp}} \subset \mathcal{B}\otimes_\mathcal{A}^{\mathrm{opp}}.$$
This implies that $\mathcal{C} \subset \mathcal{B}$ and therefore $\mathcal{C}=\mathcal{B}$.
 \end{proof}

 \begin{thm}
\label{herst}
Let $\mathcal{B}$ be a simple $k$-subalgebra of $\mathcal{A}$.

Then its centralizer $\mathcal{Z}_{\mathcal{A}}(\mathcal{B})$ is also a simple
$k$-algebra. In addition,
$$\dim_k(\mathcal{B})\cdot  \dim_k (\mathcal{Z}_{\mathcal{A}}(\mathcal{B}))=\dim_k(\mathcal{A}).$$
\end{thm}

\begin{proof}
This is a special case of Theorem 4.3.2 on p. 104 of
\cite{Herstein}
\end{proof}

\begin{sect}
\label{rank}
Iy is well known that $\dim_k({\mathcal{A}})$ is a square. Let us put
$$d=d_{\mathcal{A}}:=\sqrt{\dim_k(\mathcal{A})}.$$

Let $k_0$ be a subfield of $k$ such that $k/k_0$ is a finite algebraic {\sl separable} field extension.
Let $\bar{k}_0$ be an algebraic closure of $k_0$. We write $\Sigma_k$ for the $[k:k_0]$-element set of $k_0$-linear field embeddings $k\hookrightarrow \bar{k}_0$. It is well known that
the canonical homomorphism of semisimple commutative $\bar{k}_0$-algebras
$$k\otimes_{k_0}\bar{k}_0 \to \oplus_{\sigma\in \Sigma_k}k\otimes_{k,\sigma}\bar{k}_0$$
is an isomorphism. Notice also that each $k\otimes_{k,\sigma}\bar{k}_0$ is canonically isomorphic to $\bar{k}_0$. This implies easily that
the canonical homomorphism of semisimple  $\bar{k}_0$-algebras
$$\mathcal{A}\otimes_{k_0}\bar{k}_0 \to \oplus_{\sigma\in \Sigma_k}\mathcal{A}\otimes_{k,\sigma}\bar{k}_0$$
is an isomorphism. In addition, each $\mathcal{A}\otimes_{k,\sigma}\bar{k}_0$ is isomorphic to the matrix algebra $\M_d(\bar{k}_0)$ of size $d$ over $\bar{k}_0$.
This implies that $\mathcal{A}\otimes_{k_0}\bar{k}_0$ is isomorphic to a direct sum of $[k:k_0]$ copies of  $\M_d(\bar{k}_0)$.
\begin{rem}
Suppose that $\fchar(k_0)=0$ and provide $\mathcal{A}$ with the structure of the (reductive) $k_0$-Lie algebra, defining
$$[u,v]=uv-vu \ \forall u,v \in \mathcal{A}.$$
Then  $[k:k_0] d_{\mathcal{A}}$ is the  rank  $\rk(\mathcal{A}/k_0)$ of the reductive $k_0$-Lie algebra  $\mathcal{A}$. Indeed, the rank of the
$k_0$-Lie algebra  $\mathcal{A}$ coincides with the rank of the $\bar{k}_0$-Lie algebra  $\mathcal{A}\otimes_{k_0}\bar{k}_0$
while the latter equals $[k:k_0]$ times the rank of $\M_d(\bar{k}_0)$. It remains to recall that the rank of $\M_d(\bar{k}_0)$ over $\bar{k}_0$  equals $d=d_{\mathcal{A}}$.
\end{rem}

\begin{thm}
\label{mainalg}
  Let $\E$ be a
  subfield of $\mathcal{A}$  such that $\E\supset k_0$. (In particular, $\mathcal{A}$ and $\E$ have the same multiplicative identity $1$.)
   Let $k \E\subset \mathcal{A}$ be the image of the natural $k$-algebra homomorphism
$$\E\otimes_{k_0}k \to \mathcal{A}, \ u\otimes c\mapsto uc=cu \ \forall u\in \E, c\in k.$$
and
$\mathcal{Z}_{\mathcal{A}}(\E)\subset \mathcal{A}$  the centralizer of $\E$ in $\mathcal{A}$.

Then $\E, k \E$ and $\mathcal{Z}_{\mathcal{A}}(\E)$ enjoy the following properties.

\begin{itemize}
\item[(0)]
The degree $[\E:k_0]$ divides  $\rk(\mathcal{A}/k_0)=[k:k_0]  d_{\mathcal{A}}$. In addition,
if  $k\E$ is a field then $[k\E:k_0]$ divides $[k:k_0]d_{\mathcal{A}}$,   the degree $[k\E:k]$ divides $d_{\mathcal{A}}$
and $[k\E:\E]$ divides  $[k:k_0]d_{\mathcal{A}}/[\E:k_0]$.

\item[(i)]
 $k\E$ is a commutative semisimple $k$-algebra.
 \item[(ii)]
  $\mathcal{Z}_{\mathcal{A}}(\E)$ is a semisimple $k$-algebra that coincides with the centralizer of
  $k \E$ in $\mathcal{A}$.
   \item[(iii)]
   The center of  $\mathcal{Z}_{\mathcal{A}}(\E)$  coincides with $k \E$.  The centralizer of  $\mathcal{Z}_{\mathcal{A}}(\E)$   in $\mathcal{A}$ coincides with $k \E$.
   \item[(iv)]
    $\mathcal{Z}_{\mathcal{A}}(\E)$ is a simple $k$-algebra if and only if $k \E$ is a field. (E.g., if $\E$ contains $k$.)
    \item[(v)]
    If $\fchar(k_0)=0$ then
    $$\dim_{\E}(\mathcal{Z}_{\mathcal{A}}(\E)) \le \left(\frac{d_{\mathcal{A}} [k:k_0]}{[\E:k_0]}\right)^2.$$
    \item[(vi)]
     If $\fchar(k_0)=0$ then
    the equality
 $$\dim_{\E}(\mathcal{Z}_{\mathcal{A}}(\E)) = \left(\frac{d_{\mathcal{A}} [k:k_0]}{[\E:k_0]}\right)^2$$
 holds if and only if $\E$ contains $k_0$.
\end{itemize}

\end{thm}

\begin{ex}
If $\E=k$ then $[\E:k_0]=[k:k_0]$ and $\mathcal{Z}_{\mathcal{A}}(\E)=\mathcal{A}$. Then
$$\dim_k(\mathcal{Z}_{\mathcal{A}}(\E)) =d_{\mathcal{A}}^2=\left(\frac{d_{\mathcal{A}} [k:k_0]}{[k:k_0]}\right)^2=\left(\frac{d_{\mathcal{A}} [k:k_0]}{[\E:k_0]}\right)^2.$$
\end{ex}

\begin{rem}
\label{rank2}
If $\fchar(k_0)=0$ then the ranks of the $k_0$-Lie algebra $\mathcal{A}$  and its subalgebra $\mathcal{Z}_{\mathcal{A}}(\E)$ coincide. Indeed, it suffices to check that
$\mathcal{Z}_{\mathcal{A}}(\E)$ contains a Cartan subalgebra of  $\mathcal{A}$. In order to do that, notice that $\E/k_0$ is a finite separable field extension and therefore
there is $u \in \E$ that generates $\E$ over $k_0$. Clearly, $u$ is semisimple and the centralizer of $u$ in  $\mathcal{A}$ coincides with the centralizer of $\E$, i.e., with
$\mathcal{Z}_{\mathcal{A}}(\E)$. Since $u$ is semisimple, there is a Cartan subalgebra $\mathfrak{h}$ of $\mathcal{A}$ that contains $u$. Since $\mathfrak{h}$  is commutative,
it commutes with its own element $u$ and therefore lies in $\mathcal{Z}_{\mathcal{A}}(\E)$. This ends the proof.

\end{rem}

\begin{proof}[Proof of Theorem \ref{mainalg}]
Since $k/k_0$ is separable, $\E\otimes_{k_0}k$ is isomorphic to a direct sum of fields. The same is true for its quotient $k\E$, which proves (i).
Since $k$ is is the center of $\mathcal{A}$ and $k\E$ is generated by $k$ and $\E$, the centralizer of semisimple $k$-akgebra $k\E$ coincides with the centralizer of $\E$.
Now (ii) follows from Theorem \ref{ssCenter}. Since $k\E$ is commutative, (iii) follows from (ii), thanks to Theorem \ref{ssCenter}, and (iv) follows from (ii) and (iii).

Let us prove (v) and (vi).  Recall that $\mathcal{Z}_{\mathcal{A}}(\E)=\mathcal{Z}_{\mathcal{A}}(k\E)$.

 First, assume that $k\E$ is a field.  Then
 $$[k\E:k]\cdot [k:k_0]=[k\E:k_0]=[k\E:\E]\cdot [\E:k_0], \ [\E:k_0] \le [k\E:k_0]$$
 and therefore
 \begin{equation}
 \label{ineq}
 \frac{[k\E:\E]}{[k\E:k_0]^2}=\frac{1}{[\E:k_0] [k\E:k_0]}\le \frac{1}{[\E:k_0]^2};
 \end{equation}
 the equality holds if and only if $[k\E:k_0]=[\E:k_0]$, i.e., $k\E=\E$, which means that $\E$ contains $k$.

By Theorem \ref{herst},
$$\dim_k(\mathcal{Z}_{\mathcal{A}}(\E)) =\dim_k(\mathcal{Z}_{\mathcal{A}}(k\E)) =\frac{\dim_k(\mathcal{A})}{[k\E:k]}=\frac{d_{\mathcal{A}}^2}{[k\E:k]}.$$
This implies that the $k\E$-dimension of $\mathcal{Z}_{\mathcal{A}}(\E)$ is given by the formula
$$\dim_{k\E}(\mathcal{Z}_{\mathcal{A}}(\E)) =\frac{\dim_{k\E}(\mathcal{Z}_{\mathcal{A}}(\E))}{[k\E:k]}=\frac{d_{\mathcal{A}}^2}{[k\E:k][k\E:k]}=\frac{d_{\mathcal{A}}^2}{[k\E:k]^2}.$$
It follows that the $\E$-dimension of $\mathcal{Z}_{\mathcal{A}}(\E)$ is given by the formula
$$\dim_{\E}(\mathcal{Z}_{\mathcal{A}}(\E)) =[k\E:\E]\cdot \dim_{k\E}(\mathcal{Z}_{\mathcal{A}}(\E)) =\frac{[k\E:\E]}{[k\E:k]^2}\cdot  d_{\mathcal{A}}^2=$$
$$\frac{[k\E:\E]}{[k\E:k]^2 [k:k_0]^2}\cdot [k:k_0]^2 d_{\mathcal{A}}^2=\frac{[k\E:\E]}{[k\E:k_0]^2}\cdot ([k:k_0] d_{\mathcal{A}})^2 \le$$
$$ \frac{1}{[\E:k_0]^2}\cdot ([k:k_0] d_{\mathcal{A}})^2;$$
in light of \eqref{ineq}, the equality holds if and only if $\E$ contains $k$.

Now suppose that $k\E$ is {\sl not} a field and let us split semisimple $k\E$ into a finite direct sum
$$k\E=\oplus_{j\in J} F_j$$
of fields $F_j$. Here the set of indices $J$ is finite nonempty but {\sl not} a singleton. We write $e_j$ for the idenity element of $F_j \subset k\E$. Clearly,
\begin{equation}
\label{idemp}
e_j^2=e_j, \ \sum_{j\in J}e_j =1\in \mathcal{A}, \ e_j e_{j^{\prime}}=0 \  \forall   j\ne j^{\prime}.
\end{equation}
The map
$$i_j:\E \to F_j, \ u \mapsto e_ju=e_j u e_j$$
is a field embedding.
Let us put
$$\mathcal{A}_j=e_j \mathcal{Z}_{\mathcal{A}}(\E)=e_j \mathcal{Z}_{\mathcal{A}}(\E) e_j \subset \mathcal{Z}_{\mathcal{A}}(\E)\subset \mathcal{A}.$$
Clearly, $\mathcal{A}_j$ is a central simple $F_j$-algebra and
$$\mathcal{Z}_{\mathcal{A}}(\E)=\oplus_{j\in J}\mathcal{A}_j.$$
The field embedding $i_j:\E \to F_j$ allows us to view $\mathcal{A}_j$ as $\E$-algebra. Clearly,
$$\dim_{\E}(\mathcal{Z}_{\mathcal{A}}(\E)) =\sum_{j\in J}\dim_{\E}(\mathcal{A}_j).$$
Let us put
$$d_j:=\sqrt{\dim_{F_j}(\mathcal{A}_j)};$$
all $d_j$ are positive integers.

Applying Remark \ref{rank} to $F_j$ (instead of $k$) and   $\mathcal{A}_j$ (instead of  $\mathcal{A}$), we conclude that
the rank $\rk(\mathcal{A}_j)$ of  $k_0$-Lie algebra  $\mathcal{A}_j$  is $[F_j:k_0] d_j$. This implies that the rank of the  reductive $k_0$-Lie subalgebra
$\mathcal{Z}_{\mathcal{A}}(\E)$ of  $\mathcal{A}$ is $\sum_{j\in J}[F_j:k_0]d_j$.   Remarks \ref{rank} and \ref{rank2} imply that
$$\sum_{j\in J}[F_j:k_0]d_j= [k:k_0]d_{\mathcal{A}}.$$
Applying the already proven case of (v) to  $F_j$ (instead of $k$), $\mathcal{A}_j$ (instead of  $\mathcal{A}$) and the field $i_j(E)$, we conclude that
$$\dim_{\E}(\mathcal{A}_j)=\dim_{i_j(\E)}(\mathcal{A}_j) \le \frac{([F_j:k_0]d_j)^2}{[i_j(\E):k_0]^2}=\frac{([F_j:k_0]d_j)^2}{[\E:k_0]^2}.$$
This implies that
$$\dim_{\E}(\mathcal{Z}_{\mathcal{A}}(\E)) =\sum_{j\in J}\dim_{\E}(\mathcal{A}_j) \le \frac{\sum_{j\in J}([F_j:k_0]d_j)^2}{[\E:k_0]^2}.$$
Since $J$ is {\sl not} a singleton and all $d_j$ are positive,
$$\sum_{j\in J}([F_j:k_0]d_j)^2 <  \left(\sum_{j\in J}[F_j:k_0]d_j\right)^2 = (d_{\mathcal{A}} [k:k_0])^2.$$
This implies that
$$\dim_{\E}(\mathcal{Z}_{\mathcal{A}}(\E)) <  \frac{(d_{\mathcal{A}} [k:k_0])^2}{[\E:k_0]^2},$$
which ends the proof of (v) and (vi).

It remains to prove (0). First assume that $k\E$ is a field. Then $\mathcal{Z}_{\mathcal{A}}(\E)$ is a central simple
$k\E$-algebra. Then the rank of  $k_0$-Lie algebra $\mathcal{Z}_{\mathcal{A}}(\E)$ equals $[k\E:k_0]\cdot \mathbf{d}$
where the positive integer
$$ \mathbf{d}:=\sqrt{\dim_{k\E}(\mathcal{Z}_{\mathcal{A}}(\E))}.$$
By Remark \ref{rank2}, the ranks of $\mathcal{A}$ and $\mathcal{Z}_{\mathcal{A}}(\E)$  do coincide  and therefore the rank of  $k_0$-Lie algebra
$\mathcal{A}$ is divisible by  $[k\E:k_0]$. This means that $[k:k_0] d_{\mathcal{A}}$ is divisible by $[k\E:k_0]$,
   $[k:k_0] d_{\mathcal{A}}$ is
  divisible by $[k\E:k_0]$.  Since $[k\E:k_0]=[k\E:\E] [\E:k_0]$,
  $[k\E:k]$ divides $]d_{\mathcal{A}}$
and $[k\E:\E]$ divides  $[k:k_0]d_{\mathcal{A}}/[\E:k_0]$.
  In addition, $[k\E:\E]$ divides  $[k:k_0] d_{\mathcal{A}}/[\E:k_0]$.

Now let us do the general case when (in the notation above) $k\E$ is a direct sum $\oplus_{j\in J} F_j$ of overfields $F_j\supset E$ and
$\mathcal{Z}_{\mathcal{A}}(\E)$ is a direct sum $\oplus_{j\in J}\mathcal{A}_j$ of central simple $F_j$-algebras $\mathcal{A}_j$. Then the rank of  $k_0$-Lie algebra
$\mathcal{A}_j$ equals $[F_j:k_0]\cdot  \mathbf{d}_j$
where the positive integer
$$ \mathbf{d}_j=\sqrt{\dim_{F_j}(\mathcal{A}_j)}.$$
Since $[F_j:k_0]$ is divisible by $[\E:k_0]$, the rank of  $\mathcal{A}_j$ is also  divisible by $[\E:k_0]$. Since the rank of $\mathcal{Z}_{\mathcal{A}}(\E)$ is the sum of the ranks of
$\mathcal{A}_j$, it is also divisible by  $[\E:k_0]$. By Remark \ref{rank2}, the ranks of $\mathcal{A}$ and $\mathcal{Z}_{\mathcal{A}}(\E)$  do coincide  and therefore the rank of  $k_0$-Lie algebra
$\mathcal{A}$ is divisible by  $[\E:k_0]$.
\end{proof}

\end{sect}

\begin{sect}
\label{groupA}
We write $\Aut_{k_0}(\mathcal{A})$ for the automorphism group of the (associative) $k_0$-algebra $\mathcal{A}$. Let $G$ be a group and
$$\rho: G \to \Aut_{k_0}(\mathcal{A})$$
be a group homomorphism. Clearly, $k_0$ lies in the subalgebra $\mathcal{A}^G$ of $G$-invariants of $\mathcal{A}$. It is also clear that $G$ leaves stable the center $k$, i.e., $\rho$ induces
the group homomorphism
$$\rho_k: G \to \Aut(k/k_0)$$
where $\Aut(k/k_0)$ is the (finite) automorphism group of the field extension $k/k_0$.

\begin{thm}
\label{invariants}
Suppose that $\E$ is a field that lies in $\mathcal{A}^G$  and contains $k_0$. Then $\E$ and $\mathcal{Z}_{\mathcal{A}}(\E)$ enjoy the following properties.

\begin{itemize}
\item[(i)]
The field $\E$ is a finite algebraic  extension of $k_0$ and the degree $[\E:k_0]$  divides $\rk(\mathcal{A}/k_0)=[k:k_0]d_{\mathcal{A}}$.
\item[(ii)]
The subalgebras $k\E$  and  $\mathcal{Z}_{\mathcal{A}}(\E)$ of $\mathcal{A}$ are $G$-stable.

\item[(iii)]
Let us assume that (in the notation above) $k\E$
 is a finite direct sum $\oplus_{j\in J} F_j$ of overfields $F_j\supset \E$ and
$\mathcal{Z}_{\mathcal{A}}(\E)$ is a finite direct sum $\oplus_{j\in J}\mathcal{A}_j$ of central simple $F_j$-algebras $\mathcal{A}_j=e_j \mathcal{Z}_{\mathcal{A}}(\E)$.
Then there is a group homomorphism
$$\rho_J: G \to \mathrm{Perm}(J)$$
of $G$ into the group $ \mathrm{Perm}(J)$ of permutations of $J$  such that if $\rho_J(j)=j^{\prime}$ then
$$\rho(g)(F_j)=F_{j^{\prime}},  \rho(g)(\mathcal{A}_j)=\mathcal{A}_{j^{\prime}} \ \forall g\in G.$$

\item[(iiibis)]
If $\mathcal{Z}_{\mathcal{A}}(\E)^G=\E$ then
 the action of $G$ on $J$ is transitive; in particular, for each $j,j^{\prime}
 \in J$
  there is a $k_0$-linear field isomorphism
$F_j \cong F_{j^{\prime}}$ that extends to an isomorphism of $k_0$-algebras $\mathcal{A}_j \cong \mathcal{A}_{j^{\prime}}$.
In particular, positive integers
$$\mathbf{e}_{\E}=[F_j:\E],  \ \mathbf{d}_{\E}=\sqrt{\dim_{F_j}(\mathcal{A}_j )}$$
do not depend on a choice of $j$ and
$$[k:k_0] d_{\mathcal{A}}=|J| \mathbf{e}_{\E}  \mathbf{d}_{\E} [\E:k_0].$$
Here $|J|$ is the cardinality of $J$.

\item[(iv)]
If $\mathcal{Z}_{\mathcal{A}}(\E)^G=\E$ and
 $G$ does not contain a proper subgroup with finite index dividing $([k:k_0] d_{\mathcal{A}})/ [\E:k_0]$    then $J$ is a singleton,
$k\E$ is a field and $\mathcal{Z}_{\mathcal{A}}(\E)$  is a central simple $k\E$-algebra.

\item[(v)]
If  $\mathcal{Z}_{\mathcal{A}}(\E)^G=\E$ and
$k\E$ is a field then $k\E/\E$ is a finite Galois field extension, whose degree $[k\E:\E]$   divides  $([k:k_0] d_{\mathcal{A}})/[\E:k_0]$.
In addition, $\rho_k$ induces the surjective group homomorphism
$$\rho_{k\E}: G \twoheadrightarrow  \Gal(k\E/\E).$$
In particular, if $G$ does not admit a proper normal subgroup with finite index dividing  $([k:k_0] d_{\mathcal{A}})/[\E:k_0]$  then $k\E=\E$, i.e.,
$\E$ contains $k$.
\end{itemize}
\end{thm}

\begin{proof}

(i) follows from  the inclusion $k_0 \subset \E$ and Theorem \ref{mainalg}(0).

(ii) is obvious.

Let us prove (iii). The set $\{\mathcal{A}_j\mid j\in J\}$ is the set of (nonzero) minimal two-sided ideals of $\mathcal{A}$. Therefore  $G$ permutes elements of this set, i.e,
 there is the group  homomorphism
$$\rho_J: G \to \mathrm{Perm}(J)$$
of $G$ into the group $ \mathrm{Perm}(J)$ of permutations of $J$  such that if  $g\in G$ and $\rho_J(g)(j)=j^{\prime}$ then
$\rho(g)(\mathcal{A}_j)=\mathcal{A}_{j^{\prime}}$. Since $F_j$ (resp. $F_{j^{\prime}}$ ) is the center of  $\mathcal{A}_j$ (resp. of  $\mathcal{A}_{j^{\prime}}$)
with identity element $e_j$ (resp. $e_{j^{\prime}}$),
\begin{equation}
\label{Gact}
\rho(g)(F_j)=F_{j^{\prime}}, \  \rho(g)(e_j)=e_{j^{\prime}}.
\end{equation}

Let us prove (iiibis).
We need to check  the transitivity of the $G$-action on $J$. Notice that for each nonempty $G$-invariant subset $T\subset J$ the sum
$e_T=\sum_{j\in T}e_j$ is a  nonzero  element of $\mathcal{A}$  that is $G$-invariant, thanks to\eqref{Gact}. This implies
that $e_T$ is a nonzero element of $\mathcal{Z}_{\mathcal{A}}(\E)^G=\E$.
If the action onf $G$ on $J$ is {\sl not} transitive then $J$ is not a singleton and there exist two disjoint $G$-orbits $T_1,T_2 \subset J$. It follows from \eqref{idemp} that
$e_{T_1}e_{T_2}=0$. Since both factors are nonzero elements of the {\sl field} $\E$, we get a desired contradiction that proves  the transitivity. This proves (iiibis).

(iv) follows readily from the transitivity of the $G$-action on $J$.

Let us prove (v). So,  $k\E$ be a field. Then $k\E/\E$ is a finite algebraic field extension and it follows from Theorem \ref{mainalg}(0) that
 $[k\E:\E]$   divides  $([k:k_0] d_{\mathcal{A}})/[\E:k_0]$.
 Clearly, $k\E$ is $G$-stable and the subfield $(k\E)^G$ of its $G$-invariants coincides with $\E$. This gives us the natural group homomorphism
 $$\rho_{k\E}: G \to \Aut(k\E/\E),$$
 whose image $H:=\rho_{k\E}(G)\subset  \Aut(k\E/\E)$ is a finite group (whose order does not exceed $[k\E:\E]$. Since the subfield of $H$-invariants
 $$(k\E)^H=(k\E)^G=\E,$$
 the order of $H$ coincides with $[k\E:\E]$, the field extension $k\E/\E$ is Galois  with Galois group $H$.  Since the group homomorphism
$\rho_{k\E}: G \to H$ is surjective,  its kernel  $\ker(\rho_{k\E})$  is a normal subgroup in $G$ of index $[k\E:\E]$. This implies that $\ker(\rho_{k\E})$ is a normal subgroup of $G$, whose index
divides $([k:k_0] d_{\mathcal{A}})/[E:k_0]$. Therefore, if $G$ does not admit a proper normal subgroup with finite index dividing  $([k:k_0] d_{\mathcal{A}})/[E:k_0]$  then
$G= \ker(\rho_{k\E})$ and therefore $[k\E:\E]=1$, i.e., $k\E=\E$, which means that $\E$ contains $k$.

\end{proof}
\end{sect}

\begin{sect}
In this subsection we assume that $\mathfrak{A}$ is a semisimple finite-dimensional algebra over a field $k_0$ of characteristic zero. Then $\mathfrak{A}$
splits into a finite direct
sum
$$\mathfrak{A}= \oplus_{s\in \mathfrak{I}(\mathfrak{A})} \mathcal{A}_s$$
 of simple $k_0$-algebras $\mathcal{A}_s$. (Here the finite nonempty set  $\mathfrak{I}(\mathfrak{A})$ is identified with the
 set of (nonzero) minimal two-sided ideals in $\mathfrak{A}$.)

 \begin{ex}
 \label{endX}
 If $k_0=\Q$ and $\mathfrak{A}=\End^0(X)$ then $ \mathfrak{I}(\End^0(X))=\I(X)$.
 \end{ex}

 Let $G$ be a group and
$$\rho: G \to \Aut_{k_0}(\mathfrak{A})$$
be a group homomorphism.  Clearly, $\rho$ induces the action of $G$ on $\mathfrak{I}(\mathfrak{A})$ such that
$$\rho(g)\mathcal{A}_s=\mathcal{A}_{gs} \ \forall g \in G, \ s \in \mathfrak{I}(\mathfrak{A}).$$
Let $\E$ be a subfield of $\mathfrak{A}$ that contains $k_0$ and lies in the subalgebra  $\mathfrak{A}^G$ of $G$-invariants.
Then the centralizer $\mathcal{Z}_{\mathfrak{A}}(\E)$ of $\E$ in  $\mathfrak{A}$ is $G$-stable.

\begin{lem}
\label{transA}
Let us assume that the subalgebra $\mathcal{Z}_{\mathfrak{A}}(\E)^G$ of $G$-invariants of $\mathcal{Z}_{\mathfrak{A}}(\E)$  is a field.  Then the action of $G$ on $\mathfrak{I}(\mathfrak{A})$  is transitive.
 In particular,   simple $k_0$-algebras  $\mathcal{A}_s$ and  $\mathcal{A}_t$ are isomorphic for each pair $s,t \in \I(\mathfrak{A})$.
\end{lem}
\begin{proof}
We use the same idea as in the proof of Theorem \ref{invariants}(iii).
Let
$$e_s \in \mathcal{A}_t\subset  \sum_{t\in \mathfrak{I}(\mathfrak{A})} \mathcal{A}_t=\mathfrak{A}$$
be the identity element of $\mathcal{A}_s$.  Clearly,  $e_s$ lies in the center of  $\mathfrak{A}$ and
$$\rho(g)e_s=e_{gs} \ \forall  g \in G, \ s \in \I(\mathfrak{A}).$$
It is also clear that $e_s e_t=0$ for distinct elements $s$ and $t$ of  $\I(\mathfrak{A})$.
Notice that for each nonempty $G$-invariant subset $T\subset  \mathfrak{I}(\mathfrak{A})$ the sum
$e_T=\sum_{t\in T}e_t$ is a  nonzero  {\sl central} element of $\mathfrak{A}$  that is $G$-invariant. This implies
that $e_T$ is a nonzero element of $\mathcal{Z}_{\mathfrak{A}}(E)^G$.
If the action on $G$ on $\I(\mathfrak{A})$ is {\sl not} transitive then $J$ there exist two disjoint $G$-orbits $T_1,T_2 \subset \mathfrak{I}(\mathfrak{A})$. Clearly,
$e_{T_1}e_{T_2}=0$. Since both factors are nonzero  elements of the {\sl field}  $\mathcal{Z}_{\mathfrak{A}}(E)^G$, we get a desired contradiction that proves  the transitivity.
\end{proof}

\begin{cor}
\label{commE}
We keep the notation and assumptions of Lemma \ref{transA}.
Suppose that $K_a$ is an algebraically closed field of characteristic $0$ that contains $k_0$ and we are given a nonempty family $\{\mathcal{M}_{\tau}\mid \tau \in \Sigma\}$
of finite-dimensional $K_a$-vector spaces $\mathcal{M}_{\tau}$ that enjoy the following properties.

\begin{itemize}
\item[(i)]
Not all $\mathcal{M}_{\tau}= \{0\}$.
\item[(ii)]
For each $\tau\in \Sigma$ we are given a homomorphism of $k_0$-algebras
$$\mathcal{Z}_{\mathfrak{A}}(\E) \to \End_{K_a}(\mathcal{M}_{\tau})$$
that sends $1$ to the identity automorphism of $\mathcal{M}_{\tau}$.
\end{itemize}


If the  largest common divisor of all  $\dim_{K_a}(\mathcal{M}_{\tau})$ is $1$ then
$\mathcal{Z}_{\mathfrak{A}}(\E) $  is a finite-dimensional semisimple commutative $\E$-algebra, which is either a field or isomorphic
to a direct sum of finitely many copies of the same field.
\end{cor}

\begin{proof}

Applying Lemma \ref{transA} to the semisimple $\E$-algebra $\mathcal{Z}_{\mathfrak{A}}(\E)$ (instead of the
$k_0$-algebra $\mathfrak{A}$), we obtain that $\mathcal{Z}_{\mathfrak{A}}(\E)$ is isomorphic to a direct sum
of copies of a certain finite-dimensional simple $\E$-algebra say, $\mathcal{B}$. The center $F$ of $\mathcal{B}$
is an overfield of $\mathcal{E}$ and the field extension $F/\E$ is finite algebraic. As
 usual,
$$d_{\mathcal{B}}=\sqrt{\dim_F(\mathcal{B})}$$
is a positive integer.  This implies that the tensor product $\mathcal{B}\otimes_{k_0}K_a$ is isomorphic as a $K_a$-algebra
to a direct sum of $[\E:k_0]$ copies of the matrix algebra $\M_{d_{\mathcal{B}}} (K_a)$ of size $d_{\mathcal{B}}$ over $K_a$. This implies
that $\mathcal{Z}_{\mathfrak{A}}\E) \otimes_{k_0}K_a$ is isomorphic  as a $K_a$-algebra
to a direct sum of  copies of $\M_{d_{\mathcal{B}}} (K_a)$. On the other hand,
each  $\mathcal{M}_{\tau}$ carries the natural structure of
$\mathcal{Z}_{\mathfrak{A}}\E) \otimes_{k_0}K_a$-module. Since the $K_a$-dimension of every finite-dimensional
$\M_{d_{\mathcal{B}}} (K_a)$-module is divisible by $d_{\mathcal{B}}$, all $\dim_{K_a}(\mathcal{M}_{\tau})$
are divisible by  $d_{\mathcal{B}}$. This implies that  $d_{\mathcal{B}}=1$, i.e.,
$\mathcal{B}=F$ is a field.
\end{proof}

\end{sect}

\section{Abelian varieties and centralizers}
\label{abelianMult}
In this section we are going to prove Theorems \ref{SCenter}, \ref{boundDim} and \ref{bigDim}.
We will use Theorem  \ref{mainalg} in order to prove  Theorem \ref{inequality} below  that is a special case of these Theorems.  Later we deduce from Theorem \ref{inequality} the general case.

\begin{thm}
\label{inequality}
Suppose that $Y$ is a positive-dimensional abelian variety over $K_a$ that enjoys the following equivalent properties.

\begin{itemize}
\item[(a)]
$\End^0(Y)$ is a simple $\Q$-algebra.
\item[(b)]
The center $C_Y$ of $\End^0(Y)$ is a number field and
$\End^0(Y)$ is a central simple algebra over $C_Y$.
\item[(c)]
There exists a simple abelian variety $Z$ over $K_a$ such that $Y$ is isogenous over $K_a$
to a self-product of $Z$.
\end{itemize}
Let $E$ be a number field and $i: E \hookrightarrow \End^0(Y)$  be a $\Q$-algebra embedding. Then the $E$-algebra $\End^0(Y,i)$ enjoys the following properties.
\begin{enumerate}
\item[(i)]
$\End^0(Y,i)$ is semisimple.
\item[(ii)]
$\End^0(Y,i)$ is simple if and only if $i(E)C_Y$ is a field
\footnote{Last sentences of \cite[Remark 4.1]{ZarhinL} and  \cite[Remark 3.1]{ZarhinMZ2} wrongly assert the simplicity of  $\End^0(Y,i)$ without assuming that $i(E)C_Y$ is a field. The mistake was caused by improper use of \cite[Theorem 4.3.2 on p. 104]{Herstein}.}
. (E.g., $C_Y\subset E$ or $E\subset C_Y$ or number fields $E$ and $C_Y$ are linearly disjoint over $\Q$.)
If this is the case then $\End^0(Y,i)$ is a central simple algebra over the field $i(E)C_Y$.
\item[(iii)]
$$\dim_E(\End^0(Y,i)) \le \left(\frac{2\dim(Y)}{[E:\Q]}\right)^2.$$
\item[(iv)]
The equality
$$\dim_E(\End^0(Y,i)) = \left(\frac{2\dim(Y)}{[E:\Q]}\right)^2$$
holds if and only if
$$\dim_{C_Y}(\End^0(Y))=\left(\frac{2\dim(Y)}{[C_Y:\Q]}\right)^2$$
and $E$ contains $C_Y$.

\end{enumerate}
\end{thm}

\begin{rem}
\label{CM}
\begin{itemize}
\item[(i)]
Suppose that $Y$ satisfies the equivalent conditions (a),(b),(c) of Theorem \ref{inequality}.
This means that there are a simple abelian variety $Z$ over $K_a$ and a positive integer $r$ such that
$Y$ is isogenous to $Z^r$ over $K_a$.
In addition, $\End^0(Z)$  is a central division $C_Y$-algebra
and $\End^0(Y)$ is isomorphic to the matrix algebra $\M_r(\End^0(Z))$ of size $r$ over $\End^0(Z)$; in particular, fields $C_Y$ and $C_Z$ are isomorphic. We have
$$\dim(Y)=r \cdot \dim(Z), \ \dim_{C_Y}(\End^0(Y))=r^2 \dim_{C_Z}(\End^0(Z)).$$
Recall that the number $$d(Z):=\sqrt{ \dim_{C_Z}(\End^0(Z))}$$ is a positive integer.

 It follows from Albert's classification \cite[Sect. 21]{MumfordAV} that
$d(Z)\cdot [C_Z:\Q]$ divides $2\dim(Z)$. This implies that
$$r\cdot d(Z)\cdot [C_Z:\Q]=\sqrt{ \dim_{C_Y}(\End^0(Y))}\cdot [C_Z:\Q],$$
which divides $2r\cdot\dim(Z)=2\dim(Y)$.
Now if we put
$$k_0=\Q, \ k=C_Y, \ \mathcal{A}=\End^0(Y)$$
then
$$[k:k_0]=[C_Y:\Q]=[C_Z:\Q], \ d_{\mathcal{A}}= \sqrt{\dim_{C_Y}(\End^0(Y))}=r \cdot d(Z)$$
and
$$[k:k_0] d_{\mathcal{A}}=[C_Y:\Q] r \cdot d(Z)=[C_Z:\Q] r \cdot d(Z),$$
which
divides $r \cdot 2\dim(Z)=2\dim(Y).$ In particular,
$$[k:k_0] d_{\mathcal{A}} \le 2\dim(Y);$$
the equality holds if and only if
$$d(Z)\cdot [C_Z:\Q]=2\dim(Z).$$
Notice that this equality is equivalent to
$$
\dim_{C_Z}(\End^0(Z))=\left(\frac{2\dim(Z)}{[C_Z:\Q]}\right)^2,$$
which, in turn, is equivalent to
\begin{equation}
\label{maxCenter}
\dim_{C_Y}(\End^0(Y))=\left(\frac{2\dim(Y)}{[C_Y:\Q]}\right)^2.
\end{equation}


 \item[(ii)]
Now assume that \eqref{maxCenter} holds.
We have
$$r\cdot d(Z)=\frac{2\dim(Y)}{[C_Y:\Q]}.$$

Let $E$ be a subfield of $\End^0(Y)$ that contains $C_Y$ and $i: E \hookrightarrow \End^0(Y)$  be the inclusion map.
 It follows from Theorems \ref{ssCenter} and \ref{herst} applied to  $\mathcal{E}=i(E)$
that $\End^0(Y,i)$  is a central simple $E$-algebra and
$$\dim_{C_Y}(\End^0(Y))=[E:C_Y]\cdot \dim_{C_Y}(\End^0(Y,i)).$$
This implies that
$$ \dim_E(\End^0(Y,i))=\frac{ \dim_{C_Y}(\End^0(Y,i))}{[E:C_Y]}=\frac{\dim_{C_Y}(\End^0(Y))}{[E:C_Y]^2}=$$
$$\frac{(2\dim(Y))^2}{[C_Y:\Q]^2 [E:C_Y]^2}=
\left(\frac{2\dim(Y)}{[E:\Q]}\right)^2.$$
\item[(iii)]
For example, let $F$ be a (maximal) subfield of $\End^0(Z)$ such that
$$C_Z\subset F, \ [F:C_Z]=d(Z)$$
and let $L/C_Z$ be a degree $r$ field extension that is linearly disjoint with $F$. Then
$E:=F\otimes_{C_Z}L$ is an overfield of $C_Z$ and
$$[E:\Q]=[E:C_Z]\cdot [C_Z:\Q]=[F:C_Z] \cdot [L:C_Z] \cdot  [C_Z:\Q]=r\cdot d(Z) \cdot [C_Z:\Q]=$$
$$\frac{2\dim(Y)}{[C_Y:\Q]}\cdot [C_Y:\Q]=2\dim(Y).$$
Let us fix an embedding
$$i_0: L \hookrightarrow \M_r(C_Y)\subset \M_r(\End^0(Z))$$
that sends $1$ to $1$. Then
$$E=F\otimes_{C_Z}L \to \M_r(\End^0(Z)), \ f\otimes l \mapsto f \cdot i_0(l)$$
is a $C_Z$-algebra homomorphism that sends $1$ to $1$. Since $E$ is a field, this homomorphism is an embedding.
It follows that $\M_r(\End^0(Z))$ contains a number field of degree $2\dim(Y)$. Since $\M_r(\End^0(Z)) \cong \End^0(Y)$, the algebra
$\End^0(Y)$ contains a number field of degree $2\dim(Y)$, i.e., $Y$ is an abelian variety of CM type over $K_a$.
\end{itemize}
\end{rem}

\begin{proof}[Proof of Theorem \ref{inequality}]
Assertions (i) and (ii) follow from Theorems \ref{herst} and \ref{ssCenter}.

In order to  prove (iii) and (iv) let us put (as in Remark \ref{CM}(i))
$$k_0=\Q, \ k=C_Y, \ \mathcal{A}=\End^0(Y).$$
Then
$$[k:k_0]=[C_Y:\Q]=[C_Z:\Q],  d_{\mathcal{A}}= \sqrt{\dim_{C_Y}(\End^0(Y))}=r\cdot d(Z)$$
and
 according to  Remark \ref{CM}(i)
$$[k:k_0] d_{\mathcal{A}} \le 2\dim(Y).$$
Now the desired result follows from Theorem \ref{mainalg}(v,vi).
\end{proof}

\begin{sect}
\label{split1}

Let $X$ be an arbitrary positive-dimensional abelian variety over $K_a$.
In this subsection we use the notation of Subsection \ref{split0}.

Let $E$ be a number field and $i: E \hookrightarrow \End^0(X)$  be a $\Q$-algebra embedding that sends $1$ to $1_X$. Then the $E$-algebra $\End^0(X,i)$ enjoys the following properties.
Let $s \in \I(X)$ and
$$\pr_s:\End^0(X) = \sum_{s\in \I(X)} D_s\twoheadrightarrow D_s$$ be the corresponding
projection map. Clearly, $\pr_s i(E)\cong E$. We write
  $D_{s,E}$ for the centralizer of $\pr_s i(E)$ in $D_s$. One may easily check that
$\End^0(X,i)=\prod_{s\in\I(X)} D_{s,E}$.
We write $i_s$ for the composition $ \pr_s i: E \hookrightarrow
\End^0(X)\twoheadrightarrow D_{s} = \End^0(X_s)$. Clearly,
$$i_s(1)=e_s= 1_{X_s}, \ D_{s,E}=\End^0(X_s,i_s), \  \End^0(X,i)=\oplus_{s\in \I(X)}\End^0(X_s,i_s).$$
In particular, the ratio
$$d_{X_s,E}=\frac{2\dim(X_s)}{[E:\Q]}$$
is a positive integer, i.e., $[E:\Q]$ divides $2\dim(X_s)$.
\end{sect}

\begin{thm}
\label{inequality2}
Suppose that $X$ is a positive-dimensional abelian variety over $K_a$.

Let $E$ be a number field and $i: E \hookrightarrow \End^0(X)$  be a $\Q$-algebra embedding that sends $1$ to $1_X$.
Then the $E$-algebra $\End^0(X,i)$ enjoys the following properties.
\begin{enumerate}
\item[(i)]
$\End^0(X,i)$ is a semisimple.
\item[(ii)]
$\End^0(X,i)$ is simple if and only if $C_X$ is a field and $i(E)C_X$ is a field.
If this is the case then $\End^0(X,i)$ is a central simple algebra over the field $i(E)C_X$.
\item[(iii)]
$$\dim_E(\End^0(X,i)) \le \left(\frac{2\dim(X)}{[E:\Q]}\right)^2.$$
\item[(iv)]
the equality
$$\dim_E(\End^0(X,i)) = \left(\frac{2\dim(X)}{[E:\Q]}\right)^2$$
holds if and only if $C_X$ is a field,
$$\dim_{C_X}(\End^0(X))=\left(\frac{2\dim(X)}{[C_X:\Q]}\right)^2$$
and $E$ contains $C_X$.
\end{enumerate}
\end{thm}

\begin{proof}
We use the notation of Section \ref{split1}.
Applying Theorem \ref{inequality}(i) to each $(X_s,i_s)$, we obtain that
$\End^0(X_s,i_s)$  are semisimple $E$-algebras. This implies that their direct sum
 $\End^0(X,i)$ is also semisimple; if it simple then $\I(X)$ is a singleton, i.e. $C_X$ is a field.
 This proves (i) while (ii) follows readily from Theorem \ref{inequality}(ii).

 Let us prove (iii) and (iv). If $\I$ is a singleton then the desired result is contained in Theorem \ref{inequality}.
 Now assume that $\I$ is {\sl not} a singleton. Applying Theorem \ref{inequality}(iii) to each $(X_s,i_s)$, we obtain that
 $$\dim_E(\End^0(X_s,i_s)) \le \frac{(2\dim(X_s))^2}{[E:\Q]^2},$$
 $$\dim_E(\End^0(X,i)) = \sum_{s\in\I}\dim_E(\End^0(X_s,i_s)) \le \sum_{s\in\I}\frac{(2\dim(X_s))^2}{[E:\Q]^2}.$$
 Since  $\I$ is {\sl not} a singleton and all $\dim(X_s)$ are positive,
$$\sum_{s\in\I}\frac{(2\dim(X_s))^2}{[E:\Q]^2}<\frac{(\sum_{s\in\I}2\dim(X_s))^2}{[E:\Q]^2}=\left(\frac{2\dim(X)}{[E:\Q]}\right)^2.$$
This ends the proof.
\end{proof}

\begin{proof}[Proof of Theorems \ref{SCenter}, \ref{boundDim} and \ref{bigDim}]
 Theorems \ref{inequality} and  \ref{inequality2} combined with Remark \ref{CM} imply readily Theorems \ref{SCenter}, \ref{boundDim} and \ref{bigDim}.
\end{proof}

\begin{proof}[Proof of Theorem \ref{endoE}]
Let us choose fields $F$ and $\mathcal{K}\subset F$ as in Remark \ref{minLambda}. Then
 $$k(\lambda)=\End_{\tilde{G}_{\lambda,X,K}}(X_{\lambda})=\End_{\tilde{G}_{\lambda,X,\mathcal{K}}}(X_{\lambda}).$$
 It follows from Lemma \ref{moriO}
 that
 $\End_{\mathcal{K}}(X,i)=i(\OC)$ and therefore $\End_{\mathcal{K}}^0(X,i)=i(E)$.
 By Remark \ref{minLambda}, $\Gal(F/\mathcal{K})$ acts on $\End(X,i)$ in such a way that
 $$\End(X,i)^{\Gal(F/\mathcal{K})}=\End_{\mathcal{K}}(X,i)=i(O).$$
 Extending the action of  $\Gal(F/\mathcal{K})$ by $\Q$-linearity on $\End(X,i)\otimes\Q$,
 we get the group homomorphism
 $$\Gal(F/\mathcal{K}) \to \Aut_{\Q}(\End(X,i)\otimes\Q)=\Aut_{\Q}(\End^{0}(X,i))$$
 such that the subalgebra of $\Gal(F/\mathcal{K})$-invariants
 $$(\End^{0}(X,i))^{\Gal(F/\mathcal{K}) }=(\End(X,i))^{\Gal(F/\mathcal{K})}\otimes\Q=i(\OC)\otimes\Q=i(E)$$
 is a field.
 Applying Example \ref{endX} and Lemma \ref{transA} to $k_0=\Q, G=\Gal(F/\mathcal{K})$ and $\mathfrak{A}=\End^0(X)$,
 we conclude that $\Gal(F/\mathcal{K})$ acts transitively on $\I(X)$. This implies that all the $X_s$'s are Galois-conjugate abelian  subvarieties of $X$.
 In particular, $\dim(X_s)$ does not depend on $s$ and
 $$\dim(X)=|\I(X)| \cdot \dim(X_s).$$
  On the other hand, the results of Section \ref{split1} tell us that
 $[E:\Q]$ divides $2\dim(X_s)$. This implies that $2\dim(X)$ is divisible by $|\I(X)| [E:\Q]$
 and therefore $|\I(X)| $ divides the ratio
 $$\frac{2\dim(X)}{[E:\Q]}=d_{X,E}.$$
 The transitivity of the action of  $\Gal(F/\mathcal{K})$ on $\I(X)$ implies that the stabilizer $\Gal(F/\mathcal{K})_s$ of any $s$
 is a subgroup in $\Gal(F/\mathcal{K})$, whose index divides $d_{X,E}$. However, the conditions of Theorem \ref{endoE} imposed on $\tilde{G}_{\lambda,X,K}$
 combined with Remark  \ref{minLambda} imply that such a subgroup must coincide with the whole group $\Gal(F/\mathcal{K})$,
 i.e., $\I(X)$ is a singleton and $\End^0(X)$ is a simple $\Q$-algebra. In particular, the center $C_X$ is a field.

  By  Remark \ref{CM}(i) applied to $Y=X$,  the product $[C_X:\Q] d_{\End^0(X)}$ divides $2\dim(Y)$. Applying Theorem \ref{invariants}(iiibis and iv) to
  $$k_0=\Q, k=C_X, \mathcal{A}=\End^0(X),    \ G=\Gal(F/\mathcal{K})$$
  $\E=i(E)$ and its centralizer  $\mathcal{Z}_{\mathcal{A}}(i(E))=\End^0(X,i)$,
  we conclude that $\End^0(X,i)$ is a central simple $i(E)$-algebra provided that  the only subgroup of $\Gal(F/\mathcal{K})$, whose index divides
  $M=[C_X:\Q] d_{\End^0(X)}/[i(E):\Q]$ is the whole $\Gal(F/\mathcal{K})$. However, $M$ obviously divides $d_{X,E}$ and we have already seen that
  the only subgroup  of $\Gal(F/\mathcal{K})$, whose index divides $d_{X,E}$ is the whole $\Gal(F/\mathcal{K})$.  This ends the proof.
  \end{proof}

  \section{Tangent spaces}
  \label{tangentS}

   The aim of this section is to obtain
   an additional information about endomorphiam algebras of abelian varieties $X$  with multiplications by a number field $E$, using
   the action of  $E$ on the Lie algebra of $X$.
   
   Throughout this section $K$ is a field of characteristic $0$. 

   \begin{sect}
   \label{SigmaE}
   Let $E$ be a number field and $\Sigma_E$  be the set of field embeddings
 $\tau: E \hookrightarrow K_a$.
   To each $\tau\in \Sigma_K$ corresponds the natural surjective $K_a$-algebra homomorphism
 $$\pi_{\tau}: E\otimes_{\Q}K_a \twoheadrightarrow E\otimes_{E,\tau}K_a =:K_{a,\tau}=K_a.$$
 Taking the direct sum of all $\pi_{\tau}$'s, we get the canonical isomorphiam of $K_a$-algebras
 $$\Pi:  E\otimes_{\Q}K_a \cong \oplus_{\tau\in \Sigma_E} K_{a,\tau}.$$

 \begin{rem}
 \label{EK}
 Suppose that $\tau(E)\subset K$ for all $\tau \in \Sigma_K$. (E.g., this condition holds if $E$ is normal over $\Q$
 and $K$ contains a subfield isomorphic to $E$.) Then  to each $\tau\in \Sigma_K$ corresponds the natural surjective $K$-algebra homomorphism
 $$\pi_{\tau,K}: E\otimes_{\Q}K \twoheadrightarrow E\otimes_{E,\tau}K =:K_{\tau}=K.$$
 Taking the direct sum of all $\pi_{\tau,K}$'s, we get the canonical isomorphism of $K$-algebras
 $$\Pi_K:  E\otimes_{\Q}K \cong \oplus_{\tau\in \Sigma_E} K_{\tau}.$$
 \end{rem}

 If $\mathcal{M}$ is any $E\otimes_{\Q}K_a$-module 
  then we write for each $\tau \in \Sigma_K$
 $$M_{\tau}=\{x\in \mathcal{M}\mid u(x)=\tau(u)x \ \forall u \in E=E\otimes 1 \subset E\otimes_{\Q}K_a\}.$$
 Clearly, $M_{\tau}=K_{a,\tau}\mathcal{M}$ is an $E\otimes_{\Q}K_a$-submodule of $\mathcal{M}$ and
 $$\mathcal{M}=\oplus_{\tau\in\Sigma_K}M_{\tau}.$$
 In particular, if $\mathcal{M}$ viewed as a vector space over $K_a=1\otimes K_a$ has finite dimension then
 $$\dim_{K_a}(\mathcal{M})=\sum _{\tau\in\Sigma_K}\dim_{K_a}(M_{\tau}).$$
 \end{sect}

  \begin{sect}

 Let $V_K$ be a smooth absolutely irreducible quasiprojective variety over $K$
 and $V=V\times_K K_a$ the correspomding variety over the algebraic closure $K_a$ of $K$.
 The Galois group $\Gal(K)$ acts naturally on
 $V_K(K_a)=V(K_a)$; the set of fixed points of this action coincides with $V_K(K)$. Further we identify $V_K(K_a)$ with its bijective image in $V(K_a)$.

  Let $P$ be a $K$-point of $V_K$, which we also view as $K_a$-point of $V$. We write $\mathbf{t}_P(V)$ for the tangent $K_a$-vector space to $V$ at $P$ and
  $\mathbf{t}_{P}(V_K)$ for the tangent $K$-vector space to $V_K$ at $P$.  The natural $K_a$-linear map \cite[Remark 6.3(iii) on p. 147]{GW}
  $$\mathbf{t}_P(V)\to \mathbf{t}_{P}(V_K)\otimes_K K_a $$
  is an isomorphism of $K_a$-vector spaces \cite[Remark 6.12(iii) on p. 152]{GW}. The Galois group $\Gal(K)$ acts by semi-linear automorphisms on $\mathbf{t}_P(V)$ and the corresponding $K$-vector subspace of $\Gal(K)$-invariants
  $$\mathbf{t}_P(V)^{\Gal(K)}= \mathbf{t}_{P}(V_K)\otimes 1=\mathbf{t}_{P}(V_K).$$

  Let $Z$ be a smooth closed $K_a$-subvariety of $V$ such that $P\in Z(K_a)$. Then the induced map of the $K_a$-vector tangent spaces
 $\mathbf{t}_{P}(Z)\to \mathbf{t}_{P}(V)$ is an embedding and we identify  $\mathbf{t}_{P}(Z)$ with its image in $\mathbf{t}_{P}(V)$.
 For each $\sigma \in \Gal(K)$ the $K_a$-vector subspace
 $$\sigma (\mathbf{t}_{P}(Z))\subset \mathbf{t}_{P}(V)$$
 coincides with the tangent space to the closed smooth subvariety $\sigma Z \subset V$ at $P\in (\sigma Z) (K_a)=\sigma (Z(K_a))$.
 (This assertion follows readily  from the classical explicit description of the tangent space \cite[Example 6.5 on p. 148]{GW}.)
  \end{sect}

 \begin{sect}
 Let $X$ be a positive-dimensional abelian variety over $K_a$ that is defined over $K$. This means that there exists an abelian scheme $X_K$ over $K$ such that $X=X_K\times_K K_a$.
 Let
 $$\ol \in X_K(K)\subset X_K(K_a)=X(K_a) $$
  be the zero of the group law on $X_K$. Let us put
 $$\Lie(X)=\mathbf{t}_{\ol}(X),  \ \Lie_K(X)=\mathbf{t}_{\ol}(X_K).$$
 By definition, $\Lie(X)$ (resp. $\Lie_K(X)$) is a $\dim(X)$-dimensional vector space over $K_a$ (resp.  over $K$) and there is the natural identification of $K_a$-vector spaces
 $$\Lie(X)=\Lie_K(X)\otimes_K K_a.$$
 If $Z\subset K_a$ is an abelian $K_a$-subvariety of $X$ then $Z(K_a)$ contains $\ol$ and we consider the $K_a$-vector subspace.
 $$\Lie(Z):=\mathbf{t}_{\ol}(Z)\subset \mathbf{t}_{\ol}(X)=\Lie(X).$$
 For each $\sigma \in \Gal(K)$ we have the abelian $K_a$-subvariety $\sigma Z$ and
 $$\Lie(\sigma Z)=\sigma (\Lie(Z))\subset \Lie_K(X)\otimes_K K_a=\Lie(X).$$
   By functoriality, $\Lie(X)$ (resp.  $\Lie_K(X)$) carries the natural
 structure of $\End(X)\otimes K_a=\End^0(X)\otimes_{\Q}K_a$-module  (resp. of $\End_K(X_K)\otimes K=\End_K^0(X_K)\otimes_{\Q}K$-module.)

 Let
 $$i: E \hookrightarrow  \End^0(X)$$
be  a $\Q$-algebra embedding that sends $1$ to $1_X$.

 In particular, $\Lie(X)$  becomes the $E\otimes_{\Q}K_a$-module.
 Let us consider  the $K_a$-vector subspace
 $$\Lie(X)_{\tau}=\{z\in \Lie(X)\mid i(e)z=\tau(e)z \ \forall e\in E\} \subset \Lie(X), \
n_{\tau}(X,i)=\dim_{K_a}(\Lie(X)_{\tau}).$$
Clearly,
$$\Lie(X)=\oplus_{\tau\in \Sigma_E}\Lie(X)_{\tau}, \ \dim(X)=\dim_{K_a}(\Lie(X))=\sum_{\tau\in \Sigma_E}n_{\tau}(X,i).$$
We write $n_{X,i}$ for the greatest common divisor of all $n_{\tau}(X,i)$. Clearly, $n_{X,i}$  is a positive integer dividing $\dim(X)$.
The subspace $\Lie(X)_{\tau}$ is $\End^0(X,i)$-invariant and carries the natural structure of $\End^0(X,i)\otimes_{\Q}K_a$-module.

From now on we assume that
$$i(E) \subset \End_K^{0}(X_K).$$

\begin{thm}
\label{Kn}
Suppose that $\fchar(K)=0$.
If  $\End_K^0(X,i)$ is a number field and  $n_{X,i}=1$
then $\End^0(X,i)$ is a semisimple commutative $E$-algebra and all its simple components are mutually isomorphic number fields.
\end{thm}

\begin{proof}
Let us put
$$k_0=\Q, \mathfrak{A}=\End^0(X),  G=\Gal(K), \Sigma=\Sigma_K, \mathcal{M}_{\tau}=\Lie(X)_{\tau}.$$
Applying Lemma \ref{transA} and Corollary \ref{commE} to $\E=i(E)$,
and
$$\mathcal{Z}_{\mathfrak{A}}(\E)=\End^0(X,i), \quad \ \mathcal{Z}_{\mathfrak{A}}(\E)^G=\End_K^{0}(X,i),$$
we obtain the desired result.
\end{proof}

\begin{cor}
\label{KnC}
Suppose that $$\fchar(K)=0, \ i(\OC)\subset \End_K(X), \ n_{X,i}=1.$$
 Let us assume that there exists a maximal ideal $\lambda$ of $\OC$ such that
 $$\End_{\tilde{G}_{\lambda,X,K}}(X_{\lambda})=k(\lambda)$$
 then
  $\End^0(X,i)$ is a semisimple commutative $E$-algebra and all its simple components are mutually isomorphic number fields.
\end{cor}

\begin{proof}
By Corollary \ref{moriE}, the condition on the centralizer implies that $\End_K^0(X,i)=i(E) \cong E$ is a number field.
Now the result follows from Theorem \ref{Kn}.
\end{proof}

\begin{sect}
\label{nI}
We continue our study of certain subspaces of $\Lie(X)$.
If $\tau\in \Sigma_E$ and $\sigma\in\Gal(K)$ then their composition
$$\sigma\tau: E \hookrightarrow K_a$$ also lies in $\Sigma_E$ and
$$\sigma(\Lie(X)_{\tau})=\Lie(X)_{\sigma\tau}\subset \Lie(X).$$
In particular,
$$n_{\tau}(X,i)=\dim_{K_a}(\Lie(X)_{\tau})=\dim_{K_a}(\Lie(X)_{\sigma\tau})=n_{\sigma\tau}(X,i),$$
i.e.,
$$n_{\tau}(X,i)=n_{\sigma\tau}(X,i) \ \forall \tau \in \Sigma_E, \sigma \in \Gal(K).$$
In addition,  suppose that  $Z\subset X$ is an abelian $K_a$-subvariety of $X$ such $\Lie(Z)$ is $E$-invariant
(i.e., is a $E\otimes_{\Q}K_a$-submodule of $\Lie(X)$). Then $\Lie(\sigma Z)$ is also $E$-invariant and
$$\sigma(\Lie(Z)_{\tau})=\Lie(\sigma Z)_{\sigma\tau}.$$
In particular, if $\tau(E)\subset K$ then $\sigma\tau=\tau$ and therefore
$$\sigma(\Lie(Z)_{\tau})=\Lie(\sigma Z)_{\tau}$$
and
$$\dim_{K_a}(\Lie(\sigma Z)_{\tau})=\dim_{K_a}(\Lie(Z)_{\tau}).$$
Now we use the notation of Subsections \ref{split0} and \ref{split1}.  Recall that $X_s\subset X$ is a positive dimensional abelian $K_a$-subvariety of $X$ for all $s\in \I(X)$.
Since $\fchar(K)=0$, the isogeny $\Pi_X$ (see Lemma \ref{remL}) induces an isomorphism of $K_a$-vector spaces
$$\Lie(X)=\oplus_{s\in\I(X)} \Lie(X_s)$$
while each subspace  $\Lie(X_s)\subset \Lie(X)$ is $E$-invariant and $\End^0(X,i)$-invariant in light of results of Subsection \ref{split1}. In addition,
the action of $E$ on $\Lie(X_s)\subset \Lie(X)$ induced by $i$ coincides with the action of $E$ induced by $i_s: E \hookrightarrow \End^0(X_s)$. This implies that
$$\dim_{K_a}(\Lie(X_s)_{\tau})=n_{\tau}(X_s,i_s) \ \forall s\in \I(X), \tau\in \Sigma_E.$$
It is also clear that
$$\sigma(\Lie(X_s))=\Lie(\sigma(X_s))=\Lie(X_{\sigma(s)}) \ \forall \sigma \in \Gal(K), s\in \I(X).$$
So, if
\begin{equation}
\label{imageK}
\tau(E)\subset K \ \forall \tau \in \Sigma_E
\end{equation}
and the action of $\Gal(K)$ on $\I(X)$ is transitive  then $\dim_{K_a}(\Lie(X_s)_{\tau})$ does {\sl not} depend on a choice of $s$ and
$$n_{\tau}(X,i)=\dim_{K_a}(\Lie(X)_{\tau})=|\I(X)| \dim_{K_a}(\Lie(X_s)_{\tau}.$$
This implies that if \eqref{imageK} holds and the Galois action on $\I(X)$ is transitive then  $n_{\tau}(X,i)$ is divisible by $|\I(X)|$  for all $\tau \in \Sigma_E$.
It follows that $n_{X,i}$ is  divisible by $|\I(X)|$.
\end{sect}

\begin{lem}
\label{tangentSN}
Suppose that $\fchar(K)=0$ and $\tau(E)\subset K$ for all $\tau\in \Sigma_E$.
If  $\End_K^0(X,i)$ is a number field and  $n_{X,i}=1$
then $\I(X)$ is a singleton, i.e., $X=X_s$,  $C_X$ is a number field and $\End^0(X)$ is simple $\Q$-algebra, which is a central simple algebra over $C_X$.
\end{lem}

\begin{proof}
If  $\End_K^0(X,i)$ is a number field then $\Gal(K)$ acts on $\I(X)$ transitively. By results of Subsection \ref{nI},
$n_{X,i}$
is divisible by $|\I(X)|$.
Since  $n_{X,i}=1$, $\I(X)$ is a singleton, i.e.,  $X=X_s$ and $\End^0(X)=\End^0(X_s)$ is a simple $\Q$-algebra.
\end{proof}

\begin{rem}
Lemma \ref{tangentSN} is a generalization of (\cite[Th. 3.12(i)]{ZarhinMZ2}, \cite[Th. 3.12(i)]{ZarhinMZ2arxiv}).
\end{rem}

\begin{thm}
\label{mainEtangent}
Suppose that
$$\fchar(K)=0, \ \End_K^0(X,i)=i(E), \ n_{X,i}=1, \ \tau(E)\subset K \  \forall \ \tau\in \Sigma_E.$$
Then $\End^0(X,i)$ is a number field containing $E$ and the degree $[\End^0(X,i):i(E)]$ divides $d_{X,E}$.
\end{thm}

\begin{proof}
Let us put $k_0=\Q$. By Lemma \ref{tangentSN},  $\mathcal{A}:=\End^0(X)$ is a central simple algebra over the number field $k:=C_X$.
Let us apply Theorem \ref{invariants} to $G=\Gal(K)$, the field $\E=i(E)$  and
$$\mathcal{Z}_{\mathcal{A}}(E)=\End^0(X,i), \  \mathcal{Z}_{\mathcal{A}}(E)^G=\End_K^{0}(X,i)=i(E).$$
By Theorem \ref{Kn}, $\End^0(X,i)$  (in the notation of  Theorem \ref{invariants})  is a direct sum of fields
$$\End^0(X,i)=\oplus_{j\in J}F_j$$
where all $F_j$'s are mutually isomorphic number fields. By  Theorem \ref{invariants}(iii, iiibis), there is a  {\sl transitive} action
$$\rho_J: \Gal(K) \to \Perm(J)$$
of $\Gal(K)$ on $J$  such that if $\rho_J(\sigma)j=j^{\prime}$ then $\sigma(F_j)=F_{j^{\prime}}$.
Let $e_j \in F_j\in \End^0(X,i)$ be the identity element of $F_j$. Clearly,
$$\sum_{j\in J}e_j=1 \in \End^0(X), \ e_j^2=e_j^2, \ e_j e_{j^{\prime}}=0 \ \forall j \ne j^{\prime}.$$
This implies that the set $\{e_j\mid j\in J\}$ is $\Gal(K)$-invariant and the action of $\Gal(K)$ on this set is transitive. Let us put
$$\Lie(X)^{(j)}=e_j \Lie(X)\subset \Lie(X).$$
Clearly, each $\Lie(X)^{(j)}$ is a $E\otimes_{\Q}K_a$-sumbodule of $\Lie(X)$ and
$$\Lie(X)=\oplus_{j\in J}\Lie(X)^{(j)}.$$
In addition, $\Gal(K)$ acts transitively on the set $\{\Lie(X)^{(j)}\mid j\in J\}$.  Since $\tau(E)\subset K$   for each $\tau\in \Sigma_E$,
$\dim_{K_a}(\Lie(X)^{(j)}_{\tau})$ does {\sl not} depend on a choice of $j\in J$.
This implies that
$$n_{\tau}(X,i)=\dim_{K_a}(\Lie(X)_{\tau})=|J| \dim_{K_a}(\Lie(X)^{(j)}_{\tau});$$
in particular, all $n_{\tau}(X,i)$ are divisible by $|J|$. This implies that $n_{X,i}$ is divisible by $|J|$. Since  $n_{X,i}=1$,
$J$ is a singleton, i.e., $\End^0(X,i)=F_j$ is a (number) field.

 It remains to prove that
 $[F_j:E]$ divides $d_{X,E}$. Indeed,  since $F_j$ is a subfield of $\End^0(X)$, its degree $[F_j:\Q]$ divides $2\dim(X)$ and therefore
 $$[F_j:E]=\frac{[F_j:\Q]}{[E:\Q]}$$
 divides
 $$\frac{2\dim(X)}{[E:\Q]}=d_{X,E}.$$
\end{proof}

\begin{thm}
\label{mainElambda}
Suppose that $$\fchar(K)=0, \ i(\OC)\subset \End_K(X), \ n_{X,i}=1,
\ \tau(E)\subset K \  \forall \ \tau\in \Sigma_E.$$
 Let us assume that there exists a maximal ideal $\lambda$ of $\OC$ such that
 $$\End_{\tilde{G}_{\lambda,X,K}}(X_{\lambda})=k(\lambda)$$
 and $\tilde{G}_{\lambda,X,K}$ does not contain a proper normal subgroup with index dividing $d_{X,E}$.

 Then $\End^0(X,i)=i(E) \cong E$.
\end{thm}

\begin{proof}
By Corollary \ref{moriE}, the condition on the centralizer implies that
$$\left[\End^0(X,i)\right]^{\Gal(K)}=\End_K^0(X,i)=i(E).$$
Applying Theorem \ref{mainEtangent},  we conclude that $\End^0(X,i)$ is a field containing
$E$ and $\left[\End^0(X,i):E\right]$ divides $d_{X,E}$.
By Remark \ref{minLambda}, there exist a finite Galois extension $F/K$ and an overfield $\mathcal{K}$ of $K$ that is a subfield of $F$
 that enjoys the following properties.

 \begin{itemize}
 \item[(i)]
  $$\End_{\tilde{G}_{\lambda,X,\mathcal{K}}}(X_{\lambda})
  =\End_{\tilde{G}_{\lambda,X,K}}(X_{\lambda})=k(\lambda)$$
  and
  $$\tilde{G}_{\lambda,X,\mathcal{K}}=\tilde{G}_{\lambda,X,K}\subset \Aut_{k(\lambda)}(X_{\lambda}).$$
  This implies that $\End_{\mathcal{K}}^0(X,i)=i(E)$.
 \item[(ii)]
 There is a surjective group homomorphism
 $$\Gal(F/\mathcal{K}) \twoheadrightarrow \tilde{G}_{\lambda,X,\mathcal{K}}=\tilde{G}_{\lambda,X,K},$$
 which is a {\sl minimal cover}.
 In particular, $\Gal(F/\mathcal{K})$ also does {\sl not} contain a proper normal subgroup with index dividing $d_{X,E}$.
  \item[(iii)]
 The homomorphism
 $$\kappa_{X,\mathcal{K}}:\Gal(\mathcal{K}) \to \Aut(\End^0(X))= \Aut_{\Q}(\End^0(X))$$
 factors through
 $$\Gal(\mathcal{K}) \twoheadrightarrow \Gal(F/\mathcal{K}).$$
 Since $\End^0(X,i)$ is a $\Gal(K)$-stable subalgebra of $\End^0(X)$,  there is a group homomorphism
 $$\mathbf{\kappa}: \Gal(F/\mathcal{K}) \to \Aut_{\Q}(\End^0(X,i)),$$
 such that the subalgebra $\left[\End^0(X,i)\right]^{\Gal(F/\mathcal{K})}$ of $\Gal(F/\mathcal{K})$-invariants coincides with
 $$\left[\End^0(X,i)\right]^{\Gal(\mathcal{K})}=\End_{\mathcal{K}}^0(X,i)=i(E).$$
 \end{itemize}
Let $\Gamma$ be the image of
$$\mathbf{\kappa}: \Gal(F/\mathcal{K}) \to \Aut\left(\End^0(X,i)/i(E)\right).$$
 Clearly,
$$\left[\End^{0}(X,i)\right]^{\Gamma}=i(E)$$
and Galois theory tells us that  $|\Gamma|=\left[\End^0(X,i):i(E)\right]$. This implies that $\ker(\mathbf{\kappa})$ is a subgroup of index
$\left[\End^0(X,i):i(E)\right]$ in $\Gal(F/\mathcal{K})$.  This implies that the index of  $\ker(\mathbf{\kappa})$ in $\Gal(F/\mathcal{K})$
divides $d_{X,E}$ and therefore $\Gal(F/\mathcal{K})=\ker(\mathbf{\kappa})$, i.e., $\Gamma$ is the trivial group of order $1$ and
$$i(E)=\left[\End^{0}(X,i)\right]^{\Gamma}=\End^0(X,i).$$
\end{proof}

\begin{rem}
Theorem \ref{mainElambda} is a generalization of (\cite[Th. 3.12(ii)]{ZarhinMZ2}
\footnote{The assertion (ii)(a) of \cite[Th. 3.12(ii)]{ZarhinMZ2} is wrong without additional assumptions.} 
, \cite[Th. 3.12(ii)]{ZarhinMZ2arxiv}).
\end{rem}

\end{sect}

\section{Doubly Transitive Permutation Groups and Permutational Modules}
\label{permutation}
In order to apply our results to endomorphism algebras of superelliptic jacobians,
we need to discuss modular representations that correspond to permutation groups.

Let $T$ be a finite nonempty set, $n=|T|$ and $\Perm(T) \cong \ST_n$ the group of permutations of $T$. We write $\Alt(T) \cong \A_n$
for the only (normal) subgroup of index $2$ in $\Perm(T)$.

Let $\ell$ be a prime. One may attach to $T$ the following natural linear representations of $\Perm(T) $ over $\F_{\ell}$.
In what follows we assume that
$$n \ge 3.$$
First, let us consider  the space
$\F_{\ell}^T$ of  all functions $\phi: T \to \F_{\ell}$. The action of $\Perm(T)$ on $T$ gives rise to the faithful $n$-dimensional linear 
representation
$$\Perm(T) \to \Aut_{\F_{\ell}}(\F_{\ell}^T).$$
More precisely, each $g \in \Perm(T)$ sends a function $\phi: T \to \F_{\ell}$ to the function
$$[g]\phi:t \mapsto \phi(g^{-1}t) \ \forall t\in T.$$
The representation space $\F_{\ell}^T$ contains the invariant line $\F_\ell\cdot 1_T$ of constant functions (where $1_T$ is the constant function $1$)
and the invariant $(n-1)$-dimensional  hyperplane of functions with zero ``integral''
$$(\F_{\ell}^T)^{0}=\{\phi: T \to \F_{\ell}\mid \sum_{t\in T}\phi(t)=0\}\subset \F_{\ell}^T.$$
Clearly,
$$\F_\ell\cdot 1_T=(\F_{\ell}^T)^{\Perm(T)},$$
i.e., $\F_\ell\cdot 1_T$ is the subspace of $\Perm(T)$-invariants in $\F_{\ell}^T$.

If $\ell$ does not divide $n$ then
$$\F_{\ell}^T=\F_\ell\cdot 1_T\oplus (\F_{\ell}^T)^{0}.$$
This implies that if $\ell$ does {\sl not} divide $n$ then $(\F_{\ell}^T)^{0}$ is a {\sl faithful} $\Perm(T)$-module.

If $\ell$ divides $n$ then $\F_\ell\cdot 1_T\subset (\F_{\ell}^T)^{0}$ and we may get the {\sl heart} of the permutational representation \cite{Mortimer}
$$ (\F_{\ell}^T)^{00}=(\F_{\ell}^T)^{0}/(\F_\ell\cdot 1_T),$$
which also carries the natural structure of $(n-2)$-dimensional representation space
$$\Perm(T) \to \Aut_{\F_{\ell}}((\F_{\ell}^T)^{00}).$$

We may also consider the quotient
$$(\F_{\ell}^T)_0=\F_{\ell}^T/(\F_\ell\cdot 1_T),$$
which is also provided with  the natural structure of $(n-1)$-dimensional representation space
$$\Perm(T) \to \Aut_{\F_{\ell}}((\F_{\ell}^T)_0)$$
\cite{Xue}.
If $\ell$ does not divide $n$ then the $\Perm(T)$-modules $(\F_{\ell}^T)^{0}$
and  $(\F_{\ell}^T)_{0}$ are canonically isomorphic. If $\ell$ divides $n$ then
$$(\F_{\ell}^T)_0=\F_{\ell}^T/(F_\ell\cdot 1_T) \supset (\F_{\ell}^T)^{0}/(F_\ell\cdot 1_T)= (\F_{\ell}^T)^{00},$$
i.e., $(\F_{\ell}^T)_0$ contains a $\Perm(T)$-invariant hyperplane that is isomorphic as $\Perm(T)$-module to $ (\F_{\ell}^T)^{00}$.

\begin{lem}
\label{faithodd}
Suppose that
$$n\ge 4, \  \ell>2, \  \ell\mid n.$$
Then both $\Perm(T)$-modules $(\F_{\ell}^T)^{00}$ and $(\F_{\ell}^T)_{0}$  are faithful.
\end{lem}

\begin{proof}
Since $(\F_{\ell}^T)^{00}$ is isomorphic to a submodule of  $(\F_{\ell}^T)_{0}$, it suffices to check the faithfulness of $Perm(T)$-module $(\F_{\ell}^T)^{00}$.
Let $g$ be a non-identity permutation of $T$. The there is $t\in T$ such that $s=g(t)\ne t$. Let $u:=g^{-1}(t)$. Clearly, $u \ne t$.
No matter whether $u$ coincides with $s$ or not,
 there exists $\phi \in (\F_{\ell}^T)^{0}$ such that
$\phi(s)=\phi(u)=1, \phi(t)=0$. (Here we use that $|T|=n>3$.)
 Then
$$[g]\phi(s)=\phi(t)=0, \ [g]\phi(t)=\phi(u)=1.$$
This implies that the function $[g]\phi-\phi$ takes values $-1$ at $s$ and $1$ at $t$. In particular, it is {\sl not} a constant function. This implies that the image of $\phi$
in $(\F_{\ell}^T)^{0}/\F_{\ell}\cdot 1_T)= (\F_{\ell}^T)^{00}$ is {\sl not} $g$-invariant.
This implies that the action of $\Perm(T)$ on $(\F_{\ell}^T)^{00}$ is faithful.
\end{proof}

\begin{lem}
\label{faitheven}
Suppose that
$$n\ge 5, \  \ell=2, \  2\mid n.$$
Then both $\Perm(T)$-modules $(\F_{2}^T)^{00}$ and $(\F_{2}^T)_0$  are faithful.
\end{lem}

\begin{proof}
Since $(\F_{2}^T)^{00}$ is isomorphic to a submodule of  $(\F_{2}^T)_{0}$, it suffices to check the faithfulness of $\Perm(T)$-module $(\F_{2}^T)^{00}$.
Since $\Alt(T)$ is a subgroup of $\Perm(T)$,  $(\F_{2}^T)^{00}$ carries the natural structure of the $\Alt(T)$ -module and it is known \cite{Mortimer} that this module is simple.
Since $\dim_{\F_{2}}((\F_{2}^T)^{00})=n-2\ge  5-2>1$, the corresponding homomorphism  $\Alt(T) \to \Aut_{\F_{2}}((\F_{2}^T)^{00})$ is nontrivial. Since
$\Alt(T) \cong \A_n$ is simple (recall that $n\ge 5$), this homomorphism must be injective.  Since $\A_n$ is the only normal subgroup of $\ST_n\cong \Perm(T)$
(except the trivial one and $\ST_n$ itself), we conclude that the group homomorphism $\Perm(T) \to \Aut_{\F_{2}}((\F_{2}^T)^{00})$ is injective, i.e., $(\F_{2}^T)^{00}$ is
a faithful $\Perm(T)$-module.
\end{proof}

\begin{rem}
The only missing cases not covered by Lemmas \ref{faithodd} and \ref{faitheven} correspond to
$n=\ell=3$ and $n=4,\ell=2$. In both cases the $\Perm(T)$-module $(\F_{2}^T)^{00}$ is  {\sl not} faithful.
\end{rem}

Let $\mathcal{G} \subset \Perm(T)$ be a permutation (sub)group. We may view $\F_{\ell}^T, (\F_{\ell}^T)^{0},(\F_{\ell}^T)^{00}, (\F_{\ell}^T)_0$ as $\F_{\ell}$-linear representations of $\mathcal{G}$.
One may easily check that the $\F_{\ell}$-dimension of the subspace $(\F_{\ell}^T)^{\mathcal{G}}$ of $\mathcal{G} $-invariants equals the number of $\mathcal{G}$-orbits in $T$. In particular,
 $(\F_{\ell}^T)^{\mathcal{G}}=F_\ell\cdot 1_T$ if and only if $G$ is transitive.

The following statement is contained in \cite[Satz 4 and Satz 11]{Klemm}. (In the notation of \cite{Klemm},
$$p=\ell, K=\F_{\ell}, \Omega=T,  M^1=(\F_{\ell}^T)_{0}, M=(\F_{\ell}^T)^{00}. \ )$$

\begin{lem}
\label{double}
\begin{itemize}
\item[(i)]
Suppose that $\ell$ does not divide $n$ and  $\mathcal{G}$ acts transitively on $T$.
Then
$\End_{\mathcal{G}}( (\F_{\ell}^T)^{0})=\F_{\ell}$ if and only if $\mathcal{G}$ is doubly transitive.
\item[(ii)]
Suppose that $\ell$  divides $n$.
If  $\mathcal{G}$ is 3-transitive then
$$\End_{\mathcal{G}}( (\F_{\ell}^T)^{00})=\F_{\ell}.$$
\item[(iii)]
Suppose that $n \ge 4$,    $\mathcal{G}$ acts transitively on $T$ and $\ell$  divides $n$. Suppose that
$\End_{\mathcal{G}}( (\F_{\ell}^T)^{00})$ is a field.
Then either $\ell=2$ and $n$ is congruent to $2$ modulo $4$ or  $\mathcal{G}$ is doubly transitive.
\end{itemize}
\end{lem}

Actually, one may remove the transitivity condition in Lemma \ref{double}(a).

\begin{cor}
\label{doubleNT}
Suppose that $\ell$ does not divide $n$. Then
$\End_{\mathcal{G}}( (\F_{\ell}^T)^{0})=\F_{\ell}$ if and only if $\mathcal{G}$ is doubly transitive.
\end{cor}

\begin{proof}
Recall that $n \ge 3$. In light of Lemma \ref{double}(a), we need to check only the transitivity of  $\mathcal{G}$  if
$\End_{\mathcal{G}}( (\F_{\ell}^T)^{0})=\F_{\ell}$.

Suppose that  $\mathcal{G}$  is {\sl not} transitive,  i.e., one may split $T$ into a disjoint union $T=T_1\cup T_2$ of two nonempty $\mathcal{G}$-stable
subsets $T_1$ and $T_2$. If we put $n_i=|T_i|$ then $n_1+n_2=n$ and both $n_i \ge 1$. Since $\ell$ does {\sl not} divide $n$, it does not divide, at least, one of $n_i$.
We may assume that $\ell$ does {\sl not} divide $n_1$. Let us consider $u \in \End_{\mathcal{G}}( (\F_{\ell}^T)^{0})$ that is defined as follows.
For each $\phi \in (\F_{\ell}^T)^{0}$
the function $u(\phi)$ takes the  value $n_1 \left (\sum_{t\in T_2}\phi(t) \right)$ at every point of $T_2$ and  takes the  value $-n_2\left (\sum_{t\in T_2}\phi(t) \right)$ at every point of $T_1$.
Clearly, the image of $u$ is the one-dimension subspace of $(\F_{\ell}^T)^{0}$ that is generated by the function
$$\psi: T \to \F_{\ell}, \  \psi(t_2)= n_1 \ \forall t_2\in T_2, \   \psi(t_1)=- n_2 \ \forall t_1\in T_1.$$
Since $\dim_{\F_{\ell}}((\F_{\ell}^T)^{0})>1$, $u$ is {\sl not} a scalar and we get  a desired contradiction.
\end{proof}

The following assertion is a special case of \cite[ Lemma 2  on p. 3]{Mortimer}.

\begin{lem}
\label{listS}
Suppose that $\ell\mid n$, $\mathcal{G}$ is transitive and the $\mathcal{G}$-module  $(\F_{\ell}^T)^{00}$
is simple. Then the list of $\mathcal{G}$ -invariant subspaces of $\F_{\ell}^T$ consists of
$\{0\}, \F_{\ell}^T, \F_\ell\cdot 1_T,  (\F_{\ell}^T)^0$.
\end{lem}

This lemma implies readily the following corollary.

\begin{cor}
\label{listM}
Suppose that $\ell\mid n$, $\mathcal{G}$ is transitive and the $\mathcal{G}$-module  $(\F_{\ell}^T)^{00}$
is simple. Then the list of $\mathcal{G}$ -invariant subspaces of $(\F_{\ell}^T)_0$ consists of
$\{0\}, (\F_{\ell}^T)^{00}, \F_{\ell}^T)_0$.

\end{cor}

\begin{thm}
\label{xueCase}
Suppose that $\ell\mid n$,  $\mathcal{G}$ is transitive  and the  $\mathcal{G}$-module  $(\F_{\ell}^T)^{00}$ is absolutely simple.
Then
$$\End_{\mathcal{G}}\left(\left (\F_{\ell}^T\right)_0\right)=\F_{\ell}.$$
\end{thm}

\begin{proof}
The absolute simplicity of $(\F_{\ell}^T)^{00}$ implies that
$$\End_{\mathcal{G}}( (\F_{\ell}^T)^{00})=\F_{\ell}.$$
Let
$$u \in \End_{\mathcal{G}}( (\F_{\ell}^T)_0).$$
We need to prove that $u \in \F_{\ell}$, i.e., $u$ is a scalar.
Then $u((\F_{\ell}^T)^{00}) \subset  (\F_{\ell}^T)_0$ is a $\mathcal{G}$-invariant subspace of $(\F_{\ell}^T)_0$ of dimension $\le n-2$.
It follows from Corollary \ref{listM} that $u((\F_{\ell}^T)^{00})\subset (\F_{\ell}^T)^{00}$. Since $\End_{\mathcal{G}}( (\F_{\ell}^T)^{00})=\F_{\ell}$,
there is $a \in \F_{\ell}$ such that the restriction of $u$ to  $(\F_{\ell}^T)^{00}$ coincides with multiplication by $a$, i.e.,
$(u-a) ((\F_{\ell}^T)^{00})=\{0\}$. Since $(\F_{\ell}^T)^{00}$ has codimension 1 in $(\F_{\ell}^T)_0$, the image
$W:=(u-a)((\F_{\ell}^T)_0)$ has dimension $ \le 1$. Since $W$ is obviously $\mathcal{G}$-stable, it follows from from Corollary \ref{listM} that $W=\{0\}$,
i.e., $u-a=0$, which in turn means that $u=a$, i.e., is a scalar. This ends the proof.
\end{proof}

\begin{ex}
Suppose that $\ell\mid n$ and $n\ge 5$. If $\mathcal{G}=\Perm(T)$ or $\Alt(T)$ then $\mathcal{G}$ is transitive and
 the  $\mathcal{G}$-module  $(\F_{\ell}^T)^{00}$ is absolutely simple \cite{Mortimer}.  By Theorem \ref{xueCase},
 $$\End_{\mathcal{G}}\left(\left(\F_{\ell}^T\right)_0\right)=\F_{\ell}.$$
 This assertion is actually contained in Lemma 3.7 of \cite[p. 339]{Xue}.
\end{ex}

 \section{Superelliptic jacobians}
\label{superE}
The aim of this section is to apply results of Section \ref{tangentS} to endomorphism algebras of superelliptic jacobians,
using group-theoretic constructions of Section \ref{permutation}.

Let $p$ be a prime, $r$ a positive integer, $q=p^r$ and $\zeta_q \in \C$ be a primitive $q$th root of unity,
$E:=\Q(\zeta_q)\subset \C$ the
$q$th cyclotomic field and $\OC:=\Z[\zeta_q]$ the ring of integers in $\Q(\zeta_q)=E$.

Let us assume that $\fchar(K)\ne p$ and $K$ contains a primitive
$q$th root of unity $\zeta$. Let $f(x) \in K[x]$ be a polynomial
of degree $n\ge 3$ without multiple roots,
 $\RR_f\subset K_a$ the ($n$-element) set of roots of $f$ and $K(\RR_f)\subset K_a$
the splitting field of $f$. We write $\Gal(f)=\Gal(f/K)$ for the
Galois group $\Gal(K(\RR_f)/K)$ of $f$; it permutes the roots of
$f$ and may be viewed as a certain permutation group of $\RR_f$,
i.e., as  a subgroup of the group $\Perm(\RR_f)\cong\Sn$ of
permutations of $\RR_f$. (The transitivity of $\Gal(f)$ is equivalent to the irreducibility of $f(x)$.)
 There is the canonical surjection
$$\Gal(K) \twoheadrightarrow \Gal(K(\RR_f)/K)=\Gal(f).$$
In particular, we may view $\Gal(f)$-modules
$$\F_p^{\RR_f},  (\F_p^{\RR_f})^{0},  (\F_p^{\RR_f})^{00}, (\F_p^{\RR_f})_0$$
as $\Gal(K)$-modules.



 Let $C_{f,q}$ be a smooth projective model of the smooth affine $K$-curve
$y^q=f(x)$.
 The map
 $(x,y) \mapsto (x, \zeta y)$
gives rise to a non-trivial birational $K$-automorphism $\delta_q:
C_{f,q} \to C_{f,q}$ of period $q$. The jacobian $J(C_{f,q})$ of
$C_{f,q}$ is an abelian variety that is defined over $K$.  By Albanese
functoriality, $\delta_q$ induces an automorphism of $J(C_{f,q})$
which we still denote by $\delta_p$. It is known (\cite[p.~149]{Poonen}, \cite[p.~458]{SPoonen}, \cite{ZarhinM,ZarhinPisa},\cite[Lemma 2.6]{Xue})
that $\delta_q$ satisfies
$$\mathcal{P}_q(\delta_q)=0 \in \End(J(C_{f,q}))$$
where the polynomial
$$\mathcal{P}_q(t)=\frac{t^q-1}{t-1}=t^{q-1}+ \dots +1 \in \Z[t].$$
Notice that
$$\mathcal{P}(t)=\prod_{j=1}^r \Phi_{p^j}(t)$$
where $\Phi_{p^j}(t)\in \Z[t]$ is the $p^j$th cyclotomic polynomial of degree $(p-1)p^{j-1}$.

Let us consider the abelian $K$-subvariety $J^{(f,q)}$ of $J(C_{f,q})$ defined as follows.
$$J^{(f,q)}=\mathcal{P}_{q/p}(\delta_q)((C_{f,q})) \subset J(C_{f,q}).$$
It is known \cite{ZarhinM,ZarhinMZ2,ZarhinPisa,Xue} that $J^{(f,q)}$ is positive-dimensional and
$J(C_{f,q})$ is $K$-isogenous to a product $\prod_{j=1}^r J^{(f,p^j)}$. E.g., if $q=p$ (i.e, $r=1$) then $J(C_{f,p})=J^{(f,p)}$.
(See also \cite{WXY}.)

Clearly, $J^{(f,q)}$ is $\delta_q$-invariant and
$$\Phi_q(\delta_q)(J^{f,q})=\{0\}.$$ This gives rise to the embedding
$$\i: \Z[\zeta_q] \to \End_K(J^{(f,q)})$$
that sends $1$ to $1_{J^{(f,q)}}$ and $\zeta_q$ to the restriction of $\delta_q$ to $J^{(f,q)}$.

 Extending $i$ by $\Q$-linearity to the $\Q$-algebra embedding
$$i:E=\Q(\zeta_q) \hookrightarrow \End_K^{0}(J^{(f,q)}),$$
which we continue to denote by $i$.    Recall that
 $$[E:\Q]=[\Q(\zeta_q):\Q]=(p-1)p^{r-1}.$$
The dimension of $J^{(f,q)}$ and $d_{J^{(f,q)},E}$ are as follows
\cite{Poonen,SPoonen,ZarhinM,ZarhinPisa,ZarhinMZ2,Xue}.

\begin{itemize}
\item[(i)]
If $p$ does {\sl not} divide $n$ then
 $$2\dim\left(J^{f,q}\right)=(n-1)(p^r-p^{r-1}),  \ d_{J^{(f,q)},E}=n-1.$$
\item[(ii)]
If $q$ divides $n$ then
 $$2\dim\left(J^{(f,q)}\right)=(n-2)(p^r-p^{r-1}),  \ d_{J^{f,q},E}=n-2.$$
 (These equalities follow from (i) combined with  \cite[Remark 4.3 on p. 352]{ZarhinM}).
\item[(iii)]
If $p$ divides $n$ but $q$ does {\sl not} divide $n$ then \cite{Xue}
$$2\dim\left(J^{(f,q)}\right)=(n-1)(p^r-p^{r-1}),  \ d_{J^{(f,q)},E}=n-1.$$
\end{itemize}

Let $\lambda$ be the maximal principal ideal $(1-\zeta_q)\Z[\zeta_q]$ in $Z[\zeta_q]=\OC$. Its residue field $k(\lambda)=\F_p$.

Here is an explicit description of the Galois module $J^{f,q}_{\lambda}$  \cite{Poonen,SPoonen,ZarhinM,ZarhinPisa,ZarhinMZ2,Xue}.

\begin{itemize}
\item[(0)] If $(n,p)$ is neither $(3,3)$ nor $(4,2)$ then
$$\tilde{G}_{\lambda,J^{(f,q)},K}\cong \Gal(f).$$
\item[(i)]
If $p$ does {\sl not} divide $n$ then $J^{(f,q)}_{\lambda}$ is isomorphic to $(\F_p^{\RR_f})^{0}$ \cite[Lemma 4.11]{ZarhinM}.
(When $p=q$ this assertion was proven in \cite{SPoonen}.)
\item[(ii)]
If $q$ divides $n$ then  $J^{(f,q)}_{\lambda}$ is isomorphic to $(\F_p^{\RR_f})^{00}$, see Theorem \ref{qdividesn} below. ( When $q=p$ this assertion was proven in \cite{Poonen}).
\item[(iii)]
If $p$ divides $n$ but $q$ does {\sl not} divide $n$ then $J^{(f,q)}_{\lambda}$ is isomorphic to $(\F_p^{\RR_f})_{0}$ \cite{Xue}.
\footnote{J. Xue \cite{Xue} assumed that $\fchar(K)=0$. However, all his arguments related to the computation of
$\dim\left(J^{(f,q)}\right)$ and
 $J^{(f,q)}_{\lambda}$
work under a weaker assumption that $\fchar(K)\ne p$.}

\end{itemize}

The results of Section  \ref{permutation} imply readily the following statement.

\begin{lem}
\label{Jlambda}
Suppose that $(n,p)$ is neither $(3,3)$ nor $(4,2)$.  Then the following conditions hold.
\begin{itemize}
\item[(A)] The group $\tilde{G}_{\lambda,J^{(f,q)},K}$ is isomorphic to $\Gal(f)$.
\item[(B)]  If $p$ does not divide $n$ and $\Gal(f)$ is doubly transitive then
$$\End_{\tilde{G}_{\lambda,J^{(f,q)},K}}(J^{(f,q)}_{\lambda})=\F_p.$$
\item[(C)]  If $q$ divides $n$ and either $\Gal(f)$ is 3-transitive or
$$\End_{\Gal(f)}((\F_p^{\RR_f})^{00})=\F_p$$
then
$$\End_{\tilde{G}_{\lambda,J^{(f,q)},K}}(J^{(f,q)}_{\lambda})=\F_p.$$
\item[(D)]  Suppose that $p$ divides $n$ but  $q$ does not divide $n$. Assume also that $\Gal(f)$ is transitive (i.e., $f(x)$ is irreducible over $K$)
and the $\Gal(f)$-module $(\F_p^{\RR_f})^{00}$ is absolutely simple.
Then
$$\End_{\tilde{G}_{\lambda,J^{(f,q)},K}}(J^{(f,q)}_{\lambda})=\F_p.$$
\end{itemize}
\end{lem}

Now let us assume that $\fchar(K)=0$.
 Here are the explicit formulas for $n_{J^{(f,q)},i}$.
Let
$$n=kq+c, \ k,c \in \Z_{+}, \ 0 \le c<q.$$

\begin{itemize}
\item[(i)]
Suppose that  $p$ does {\sl not} divide $n$, i.e., $c\ge 1$.
Then $n_{J^{(f,q)},i}$ are as follows \cite[ Sections 4 and 5, especially, Remark 4.1  and Lemma 5.1]{ZarhinMZ2,ZarhinMZ2arxiv}.
\begin{enumerate}
\item[(1)]
if $n=kq+1$ (i.e., $c=1$) then $n_{J^{(f,q)},i}=k$.
\item[(2)]
If $p$ is odd and $n-1$ is {\sl not} divisible by $q$ (i.e., $c>1$) then  $n_{J^{(f,q)},i}=1$.
\item[(3)]
If $p=2 <q$ and $n-1$ is {\sl not} divisible by $q$  (i.e., $c>1$)  then  $n_{J^{(f,q)},i}=1$ or $2$.
In addition, if either $k$ is odd or $c<q/2$ then $n_{J^{(f,q)},i}=1$.
\end{enumerate}
\item[(ii)]
Suppose that $q$ divides $n$. Then $c=0$ and
$$n-1=(k-1)q+(q-1).$$
Using  \cite[Remark 4.3 on p. 352]{ZarhinM}, and (i), we obtain the  following results similar to (i),
replacing $n$ by $n-1$, $n-1$ by $n-2$, $k$ by $k-1$ and $c$ by $q-1$ respectively.
\begin{enumerate}

\item
If $p$ is odd then $(n-2)$ is {\sl not} divisible by $q$ and  $n_{J^{(f,q)},i}=1$.
\item
If $p=2 <q$ then $n-2$ is {\sl not} divisible by $q$  and  $n_{J^{(f,q)},i}=1$ or $2$.
In addition, if  $k-1$ is odd (i.e., $k$ is even) then $n_{J^{(f,q)},i}=1$.
\end{enumerate}

\item[(iii)]
If $n \ge 5$,  $p$ divides $n$ but $q$ does {\sl not} divide $n$ then  $n_{J^{(f,q)},i}=1$ \cite[Prop. 2.2 and Remark 2.3]{Xue}.
\end{itemize}

\begin{rem}
The case of $n=3$ is discussed in \cite{ZarhinPisa,Xue2}; see also
\cite{Ribet3}.
\end{rem}

\begin{thm}
\label{centerJfq}
Suppose that $n \ge 4$ and $\fchar(K)=0$.  If $p\mid n$ then we assume additionally that $n \ge 5$.

If $\End^0(J^{(f,q)},i)$ coincides with $i(\Q(\zeta_q))=\Q[\delta_q]$
then
$$\End^0(J^{(f,q)})=\Q[\delta_q]\cong \Q(\zeta_q), \  \End(J^{(f,q)})=\Z[\delta_q]\cong \Z[\zeta_q].$$
\end{thm}

\begin{proof}
\begin{itemize}
\item [(i)]
Suppose that $p$ does {\sl not} divide $n$. Then the result is  proven in  \cite[Theorem 4.16]{ZarhinM}.
\item [(ii)]
Suppose that $q\mid n$. This case follows from (i), thanks to Remark 4.3 of \cite{ZarhinM}.
\item [(iii)]
Suppose that $p\mid n$ but $q$ does {\sl not} divide $n$. Then the result is  proven in \cite[Cor. 4.4]{Xue}
\end{itemize}
\end{proof}

\begin{thm}
\label{centerJfq2}
Suppose that $n \ge 4$
 and $(n,p)$ is
 not $(4,2)$. Assume also that  there is a
a subgroup
$$\mathcal{G}\subset \Gal(f)\subset \Perm(\RR_f)$$
such that one of the following three conditions holds.

\begin{itemize}
\item[(i)]
The prime $p$ does not divide $n$, $\mathcal{G}$
 is doubly transitive and does not contain a subgroup, whose index divides $(n-1)$ except  $\mathcal{G}$ itself.
\item[(ii)]
The prime  power $q$  divides $n$,
 $\mathcal{G}$  does not contain a proper subgroup, whose index divides $(n-2)$.
In addition, either  $\mathcal{G}$  is $3$-transitive or
$$\End_{\mathcal{G}}((\F_p^{\RR_f})^{00})=\F_p.$$
\item[(iii)]
 The prime $p$  divides $n$ but $q$ does not divide $n$.  The group  $\mathcal{G}$ is transitive
and  does not contain a proper proper subgroup, whose index divides $(n-1)$. In addition, assume that
(at least) one of the following two conditions holds.
\begin{itemize}
\item[(A3)]
The group  $\mathcal{G}$ is transitive and the
 $\mathcal{G}$-module $(\F_p^{\RR_f})^{00}$ is absolutely simple.
\item[(B3)]
The centralizer
$\End_{\mathcal{G}}\left((\F_p^{\RR_f})_0\right)=\F_p$.
\end{itemize}
\end{itemize}

Then
$$\tilde{G}_{\lambda,J^{(f,q)},K}\cong \Gal(f), \  \End_{\tilde{G}_{\lambda,J^{(f,q)},K}}(J^{(f,q)}_{\lambda})=\F_p,$$
 $\End^0(J^{(f,q)})$ is a simple $\Q$-algebra, whose center is a subfield of   $\Q[\delta_q]$, and the centralizer $\End^0(J^{(f,q)},i)$  of $\Q[\delta_q]$  in $\End^0(J^{(f,q)})$ is a central simple $\Q[\delta_q]$-algebra.
\end{thm}

\begin{rem}
 By Theorem \ref{xueCase}, the condition (A3) of Theorem \ref{centerJfq2} implies the condition (B3).
\end{rem}

\begin{proof}[Proof of Theorem  \ref{centerJfq2}]
Replacing  $K$ by its overfield  $K(\RR_f)^{\mathcal{G}}$, we may and will assume that $\Gal(f)= \mathcal{G}$.
it follows from Lemma \ref{Jlambda}   that
$$\End_{\tilde{G}_{\lambda,J^{(f,q)},K}}(J^{(f,q)}_{\lambda})=\F_p.$$
Now the desired result follows from
 Theorems  \ref{endoE}.
\end{proof}

\begin{rem}
\label{hyper}
Suppose that $q=2$, i.e.
$$\Z[\zeta_q]=\Z, \Q[\zeta_q]=\Q, \Q[\delta_q]=\Q.$$
In this case $C_{f,2}$ is a hyperelliptic curve of genus
$[(n-1)/2]$, and
$$J(C_{f,2})=J^{(f,2)},  \quad \left[\frac{n-1}{2}\right]=\dim(J(C_{f,2}))=\dim\left(J^{(f,2)}\right).$$
 Applying Theorem \ref{endoZ2}
(instead of  Theorems  \ref{endoE}), we can do slightly better. Namely,  we obtain that
$\End^0(J(C_{f,2}))$ is a central simple $\Q$-algebra
if there is a subgroup $\mathcal{G}$ of $\Gal(f)$ that enjoys the following properties.

\begin{itemize}
\item
 $\mathcal{G}$
contains neither a normal subgroup of index $2$ nor a proper subgroup
of index 
dividing
$[(n-1)/2]$.
\item
One of the following two conditions holds.

\begin{enumerate}
\item
$n$ is odd and  $\mathcal{G}$ is $2$-transitive

\item
$n$ is even
and  either $\mathcal{G}$ is $3$-transitive or
$$\End_{\mathcal{G}}((\F_p^{\RR_f})^{00})=\F_p.$$
\end{enumerate}
\end{itemize}

It follows from Albert's classification  \cite[Sect. 21]{MumfordAV} that the central simple $\Q$-algebra $\End^0(J(C_{f,2}))$ is isomorphic either to a matrix algebra over $\Q$ or to  a matrix algebra over a quaternion $\Q$-algebra.
See \cite{ZarhinMRL,ZarhinTexel,ZarhinMRL2,ZarhinBSMF,ZarhinMMJ,Elkin,ElkinZarhin,ElkinZarhin2,ZarhinP,ZarhinL,ZarhinPAMS,ZarhinT} for other results about endomorphism algebras of hyperelliptic jacobians.

\end{rem}

\begin{thm}
\label{mainJfq}
Let us assume that
$$\fchar(K)=0,  n \ge 4, q>2.$$
 If $p\mid n$ then we assume additionally that $n \ge 5$.

Suppose that
there is 
a subgroup
$$\mathcal{G}\subset \Gal(f)\subset \Perm(\RR_f)$$
such that (at least) one of the following three conditions holds.

\begin{itemize}
\item[(i)]
The prime $p$ does not divide $n$, $\mathcal{G}$
 is doubly transitive and does not contain a proper normal  subgroup, whose index divides $(n-1)$.
 Assume additionally that
$$n=kq+c, \  \ k,c \in \Z_{+}, \ 0  \le c<q.$$
where integers $p,k$ and $c$ enjoy (at least)  one of the following three properties.

 \begin{itemize}
 \item[(A1)]
 $n=q+1$, i.e., $k=1,c=1$.
 \item[(B1)]
 $p$ is odd and $c>1$ (i.e., $q$ does not divide $n-1$).
 \item[(C1)]
 $p=2<q,  c>1$ and either $k$ is odd or $c<q/2$.
 \end{itemize}

\item[(ii)]
The prime  power $q$  divides $n$,
 $\mathcal{G}$  does not contain a proper normal subgroup, whose index divides $(n-2)$.
 We also assume that
  $p$ and $k$ enjoy (at least) one of the following three properties.

 \begin{itemize}
 \item[(A2)]
 $p$ is odd.
 \item[(B2)]
 $p=2<q$  and $k$ is even.
 \item[(C2)]
 Either  $\mathcal{G}$  is $3$-transitive or
$$\End_{\mathcal{G}}\left((\F_p^{\RR_f})^{00}\right)=\F_p.$$
\end{itemize}
\item[(iii)]
The prime $p$  divides $n$ but $q$ does not divide $n$.  The group  $\mathcal{G}$
  does not contain a 
  proper
   normal  subgroup, whose index divides $(n-1)$.

In addition, assume that (at least) one of the following two conditions holds.
\begin{itemize}
\item[(A3)]
The  group $\mathcal{G}$ is transitive and the
 $\mathcal{G}$-module $(\F_p^{\RR_f})^{00}$ is absolutely simple.
 \item[(B3)]
 The centralizer
 $\End_{\mathcal{G}}\left((\F_p^{\RR_f})_0\right)=\F_p$.
\end{itemize}
\end{itemize}
Then
$$\End^0(J^{(f,q)})=\Q[\delta_q]\cong \Q(\zeta_q), \  \End(J^{(f,q)})=\Z[\delta_q]\cong \Z[\zeta_q].$$
\end{thm}

\begin{proof}
Clearly,  $(n,p)$ is neither $(3,3)$ nor $(4,2)$.
Notice that our conditions on $n$ and $q$ imply that $n_{J^{(f,q)},E}=1$.  Second, Theorem \ref{centerJfq2} implies that
$$\tilde{G}_{\lambda,J^{(f,q)},K}\cong \Gal(f),  \ \End_{\tilde{G}_{\lambda,J^{(f,q)},K}}(J^{f,q}_{\lambda})=\F_p.$$
Now
 Theorem  \ref{mainElambda} implies that  the centralizer $\End^0(J^{(f,q)},i)$ coincides with $\Q[\delta_q]=i(\Q(\zeta_q))$.
 Now the desired result follows from Theorem \ref{centerJfq}.
\end{proof}

\begin{rem}
Suppose that $\fchar(K)=0$, $n \ge 5$ and $\Gal(f)$ coincides either with the full symmetric group $\Perm(\RR_f)\cong \ST_n$ or the alternating group $\Alt(\RR_f)\cong \mathbf{A}_n$.
Then
$$\End^0(J^{(f,q)})=\Q[\delta_q]\cong \Q(\zeta_q), \  \End(J^{(f,q)})=\Z[\delta_q]\cong \Z[\zeta_q]$$
without any additional conditions on $n$ and $q$. The case when either $p$ does {\sl not} divide $n$ or $q\mid n$ was done in \cite{ZarhinM}, the case when $p\mid n$ but $q$ does {\sl not} divide $n$ was done in \cite{Xue}.
The proofs in \cite{ZarhinM} are based on the notion of a {\sl very simple} representation that was introduced in \cite{ZarhinTexel}, see also \cite{ZarhinT}.
\end{rem}

\begin{rem}
Theorem \ref{mainJfq} is a generalization of (\cite[Th. 5.2]{ZarhinMZ2}
\footnote{In Th. 5.2 of \cite{ZarhinMZ2} the assertion  (ii)(a) is actually not proven and should be ignored.}
, \cite[Th. 5.2]{ZarhinMZ2arxiv}).
\end{rem}

\section{$\delta_q$-invariant  divisors on superelliptic curves}
\label{heartqn}
The aim of this section is to construct  an isomorphism between the Galois modules $J^{(f,q)}_{\lambda}$ and $(\F_p^{\RR_f})^{00}$ 
when $q$ divides $n$. (The existence of such an isomorphism was stated and used in Section \ref{superE}.)

Suppose that $n=\deg(f)$ is divisible by $q$, i.e, there is a positive integer $m$ such that
$$n=mq.$$
We write $B=B_f$ for the set
$$B=\{(\alpha,0)\mid \alpha\in \RR_f\}\subset C_{f,q}(K_a).$$
The set $B$ consists of $\delta_q$-invariant points of $C_{f,q}(K_a)$.  Clearly,  $C_{f,q}(K_a)$ contains an affine curve
$$(C_{f,q})_0(K_a)=\{(a,b) \in K_a^2\mid f(a,b)=0\}.$$
The complement $C_{f,q}(K_a)\setminus (C_{f,q})_0(K_a)$ is a finite {\sl nonempty set}; we call its elements {\sl infinite points} of $C_{f,q}$.
The rational function $x\in K_a(C_{f,q})$ defines a finite cover $\pi: C_{f,q} \to \mathbb{P}^1$ of degree $q$. The set of branch points contains $B$ and sits in
 the (disjoint) union of $B$ and the (finite) set of infinite points of $C_{f,q}$; $\pi$ sends the latter set to the {\sl infinite  point} $\infty$ of $\mathbb{P}^1(K_a)$.
Clearly, $y$ is a {\sl local parameter} at every $P \in B$ and $\mathrm{ord}_P(x-x(P))=q$. If $\tilde{\infty}$ is any infinite point of $C$ then both
$\ord_{\tilde{\infty}}(x)$ and $\ord_{\tilde{\infty}}(y)$ are {\sl negative integers} such that $n \cdot \ord_{\tilde{\infty}}(x)=q \cdot \ord_{\tilde{\infty}}(y)$, i.e., 
$$\ord_{\tilde{\infty}}(y)=m \cdot \ord_{\tilde{\infty}}(x).$$
It follows easily from the previous remark that if $\beta \in K_a$ then the rational function $(x-\beta)\in K_a(C_{f,q})$ has a  pole at $\tilde{\infty}$, whose order does {\sl not} depend on $\beta$, including  the cases $\beta=0$ and
$\beta=\alpha\in \RR_f$.

The main result of this section is  the following statement.

\begin{thm}
\label{qdividesn}
Suppose that $n=\deg(f)$ is divisible by $q=p^r$.

Then the $\Gal(K)$-modules $J^{(f,q)}_{\lambda}$ and $(\F_p^{\RR_f})^{00}$ are isomorphic.
\end{thm}

In the course of the proof of Theorem \ref{qdividesn} we will use the following assertion that will be proven at the end of this section.

\begin{lem}
\label{leap}
Let $D=\sum_{P\in B}a_P(P)$ be a degree zero divisor with support in $B$. Then the linear equivalence class of $p^{r-1}D$ is zero if and only if there exists an integer $j$
such that all integers $a_P$'s are congruent to $j$ modulo $p$.
\end{lem}

\begin{proof}[Proof of Theorem \ref{qdividesn} (modulo Lemma \ref{leap})]
The map $P \to x(P)$ establishes a Galois-equivariant bijection between $B$ and $\RR_f$. So, it suffices to check that the Galois modules $J^{(f,q)}_{\lambda}$ and $(\F_p^{B})^{00}$ are isomorphic.
Notice that
$$J^{(f,q)}_{\lambda}=\{x \in J^{(f,q)}(K_a)\mid \delta_q(x)=x\}\subset J^{(f,q)}(K_a)=$$
$$ \mathcal{P}_{q/p}(\delta_q)((J(C_{f,q})(K_a))=\left(1+\delta_q+ \dots + \delta_q^{p^{r-1}-1}\right)(J(C_{f,q})(K_a)).$$
Since $B \subset C_{f,q}(K_a)$ consists of $\delta_q$-invariant points, the linear equivalence class of every degree zero divisor $D=\sum_{P\in B}a_P(P)$ is a $\delta_q$-invariant point of
$J(C_{f,q})(K_a)$. This implies that that the linear equivalence class of $p^{r-1}D=\sum_{P\in B}p^{r-1}a_P(P)$ lies in
$$\{x \in J^{(f,q)}(K_a)\mid \delta_q(x)=x\}=J^{(f,q)}_{\lambda}\subset J^{(f,q)}(K_a)\subset J(C_{f,q})(K_a).$$

Let us consider the following Galois-equivariant  homomorphism of $\F_p$-vector spaces
$$\Psi: (\F_p^{B})^{0} \to J^{(f,q)}_{\lambda}.$$
Let $\phi: B \to \F_p$ be a function with $\sum_{b\in B}\phi(b)=0$. We may ``lift'' $\phi$ to a map $P \to a_P\in \Z$ in such a may that
$$b_P \bmod p=\phi(P) \ \forall P\in B, \ \sum_{P\in P}a_P=0.$$
Then $D=\sum_P a_P(P)$ is a degree zero divisor on $C_{f,q}$ with support in $B$. We define $\Psi(\phi)\in J^{(f,q)}_{\lambda}$ as the linear equivalence class of $p^{r-1}D$.
First, notice that our map is well-defined. Indeed, if $P \mapsto a_P$ lifts  the zero function then all $a_P$ are divisible by $p$ and therefore  all the coefficients of $p^{r-1}D$ are divisible by $p\cdot p^{r-1}=q$.
It follows from  by Lemma \ref{leap} that the class of $p^{r-1}D$ is zero.
This proves that $\Psi$ is well-defined. Clearly, $\Phi$ is a group homomorphism and therefore is a $\F_p$-linear map. It follows from the same Lemma that
$\phi \in \ker(\Psi)$ if and only if there exists $j\in \Z$ such that all (the corresponding) $a_P$'s are congruent to $j$ modulo $p$. This means that
$$\phi(P)=j\bmod p \ \forall P\in B,$$
i.e., $\phi$ is a constant function. In other words, $\ker(\Psi)=\F_p\cdot 1_B$. Therefore $\Phi$ induces a Galois-equivariant embedding of $\F_p$-vector spaces
$$(\F_p^{B})^{00}=(\F_p^{B})^{0}/(\F_p\cdot 1_B) \hookrightarrow   J^{(f,q)}_{\lambda}.$$
This embedding is actually an isomorphism, since
$$\dim_{\F_p}((\F_p^{B})^{00})=n-2=\dim_{\F_p}(J^{(f,q)}_{\lambda}).$$
\end{proof}

It remains to prove Lemma \ref{leap}.
We will need the following  two assertions
that characterize {\sl principal divisors} with support in $B$.

\begin{lem}
\label{prinDIV}
Let
 $D=\sum_{P\in B}a_P(P)$ be a divisor on $C_{f,q}$
   with support in $B$. Then $D$ is principal
if and only if  there exist a divisor  $D_1=\sum_{P\in B}b_P(P)$ on $C_{f,q}$ with  support in $B$
and a nonnegative integer $j<q$ such that
 $m$ divides $\deg(D_1)=\sum_{P\in B}b_P$ and
$$D= q\sum_{B\in B}b_P(P) - \frac{\sum_{P\in B}b_P}{m}\left(\sum_{P\in B}(P)\right).$$
\end{lem}

\begin{cor}
\label{principalC}
Let $Q$ be a point of $B$. Then a divisor $D=\sum_{P\in B}a_P(P)$ with support in $B$ is principal if and only if there is
a degree zero divisor $D_0$ with support in $B$ and an integer $j$ such that
\begin{equation}
\label{divP}
D=q D_0 +j\left( \left(\sum_{P\in B}(P)\right)-n(Q)\right).
\end{equation}
In addition, all integers $a_P$'s are divisible by $p^{r-1}$ if and only if $j$ is  divisible by $p^{r-1}$.
\end{cor}

\begin{proof}[Proof of Lemma \ref{prinDIV}]
Suppose $D=\mathrm{div}(h)$ where $h \in K_a(C_{f,q})$ is a {\sl nonzero} rational function on $C_f$. Since $D$ is $\delta_q$-invariant,
$\delta_q^{*}h= h \delta_q$ coincides with $c\cdot h$ for some nonzrero $c\in K_a$. The $\delta_q$-invariance of the splitting
$$K_a(C_{f,q})=\oplus_{j=0}^{q-1} y^j \cdot K_a(x)$$ implies
that $h(x)=y^j\cdot u(x)$ for some nonzero rational function $u(x) \in K_a(x)$ and a nonnegative integer $j \le q-1$. It follows that all ``finite'' zeros and poles of $u(x)$ lie in $B$.
i.e.,  there exists an integer-valued function $P \mapsto b_P$ on $B$ such that $u(x)$ coincides up to multiplication by a nonzero constant to $\prod_{P \in B}(x-x(P))^{b_P}$.
Recall that the zero divisor of $y$ is $\sum_{P\in B}(P)$ while   the set of its poles coincides with the set of infinite points of $C_f$ and if $\tilde{\infty}$ is such a point then
 $$\ord_{\tilde{\infty}}(u)=(\sum_{P\in B}b_P)\ord_{\tilde{\infty}}(x)=\frac{\sum_{P\in B}b_P}{m} \cdot \ord_{\tilde{\infty}}(y).$$
 Since $h(x)=y^j u(x)$ has neither zeros nor poles at infinite points of $C_{f,q}$,
$$\frac{\sum_{P\in B}b_P}{m}+j=0.$$
On the other hand, for each $P\in B$,
$$a_P=\ord_P(h)=j+\ord_P(u)=j+q b_P.$$
This implies that
$$D=\sum_{P\in B}a_P(P)=q\sum_{P\in B}b_P(P)+j\sum_{P\in B}(P)= q\sum_{P\in B}b_P(P) - \frac{\sum_{P\in B}b_P}{m}(\sum_{P\in B}(P)).$$

Conversely, suppose that there is a divisor $\sum_{P\in B}b_P(P)$ on $C_f$ with support in $B$ such that
  $m$ divides $\left(\sum_{P\in B}b_P\right)$ and
$$D= q\sum_{P\in B}b_P(P) - \frac{\sum_{P\in B}b_P}{m}\left(\sum_{P\in B}(P)\right).$$
Clearly, $\deg(D)=0$. Let us put
$$j:=-\frac{\sum_{P\in B}b_P}{m}.$$
Let us consider the (nonzero)  rational function
$$h=y^j \prod_{P\in B}(x-x(P))^{b_P} \in K_a(C_f).$$
Clearly $h$ has neither zeros nor poles at infinite points of $C_f$, because
$$\ord_{\tilde{\infty}}(h)=j \ord_{\tilde{\infty}}(y)+(\sum_{P\in B} b_P)\ord_{\tilde{\infty}}(x)=(mj+\sum_{P\in B} b_P)\ord_{\tilde{\infty}}(x)=0\cdot \ord_{\tilde{\infty}}(x)=0.$$
This implies that the support of $\mathrm{div}(h)$ lies in $B$. For each $P\in B$
$$\ord_P(h)=j+q b_P=a_P.$$
This implies that $D=\mathrm{div}(h)$, i.e., $D$ is {\sl principal}.
\end{proof}

\begin{proof}[Proof of Corollary \ref{principalC}]
Clearly, $n(Q)-\sum_{P\in B}(P)$ is the divisor of the rational function $\left(x-x(Q)\right)^m/y$ and $q\left((P)-(Q)\right)$ is the divisor of the rational function $(x-x(P))/(x-x(Q))$.
This implies that a  divisor $D$ of the form \eqref{divP} is principal.

Conversely, suppose that a divisor $D=\sum_{P\in B}a_P(P)$ with support in $B$ is principal. Let $\sum_{P\in B}b_P(P)$ and $j$
be as in Lemma \ref{prinDIV} and its proof,  i.e.,
$$  j =- \frac{\sum_{P\in B}b_P}{m} \in\Z,  \ D=q\sum_{P\in B}b_P(P) +j\left(\sum_{P\in B}(P)\right).$$
Let us put
$$D_0=(\sum_{P\in B}b_P(P) )- (\sum_{P\in B}b_P)(Q)=(\sum_{P\in B}b_P(P) )+jm(Q).$$
Clearly, $D_0$ is a degree zero divisor with support in $B$ and
 $$D=q\sum_{P\in B}b_P(P) -q (\sum_{P\in B}b_P)(Q) +q (\sum_{P\in B}b_P)(Q)+j\left(\sum_{P\in B}(P)\right)=$$
 $$q D_0-qjm(Q)  +j\left(\sum_{P\in B}(P)\right)=q D_0-j n (Q)+j\left(\sum_{P\in B}(P)\right)=q D_0 +j\left( \left(\sum_{P\in B}(P)\right)-n(Q)\right).$$

 In order to prove the second assertion of Corollary,  notice that both  $q=p^r$  and $n=qm=p^r m$ are divisible by $p^{r-1}$ and therefore all the coefficients of $D$ are divisible by $p^{r-1}$
  if and only if all the coefficients of $j \left(\sum_{P\in B}(P)\right)$ are  divisible  by $p^{r-1}$ as well.
All the coefficients of $j \left(\sum_{P\in B}(P)\right)$ are equal to $j$ and therefore
  are divisible by $p^{r-1}$ if and only if $j$ is  divisible by $p^{r-1}$.
\end{proof}

\begin{proof}[Proof of Lemma \ref{leap}]
Let us fix a point $Q\in B$.

Suppose that the  class of $p^{r-1}D$ is zero.
By Corollary \ref{principalC} (applied to $p^{r-1}D$), there
exist a a degree zero divisor $D_0=\sum{P\in B} b_P(P)$ and an integer $j_0=j_0(Q)\in\Z$ such that
$$p^{r-1}D=p^r D_0 +p^{r-1}j_0\left( \left(\sum_{P\in B}(P)\right)-n(Q)\right).$$
This means that
$$p^{r-1}a_Q=p^r b_Q+p^{r-1}j_0(Q) \cdot (1-n), \  p^{r-1}a_P=p^r b_P+p^{r-1} j_0(Q) \ \forall P \in B\setminus \{Q\}.$$
The first equality implies that $(1-n)j_0(Q)$ is congruent to $a_Q$ modulo $p$, which means that $j_0(Q)$ is congruent to $a_Q$ modulo $p$
(since $p\mid n$). The second equality implies that $a_P$ is congruent to $j_0(Q)$ modulo $P$, i.e., $a_P$ is congruent to $a_Q$ for all $P \in B \setminus \{Q\}$.
Since $a_Q$ is obviously congruent to itself modulo $p$, we obtain that $a_P$ is congruent to $a_Q$ modulo $p$ for each $P,Q \in B$. Now we may put $j=a_Q$.

Conversely, suppose that $D=\sum_{P\in B}a_P(P)$ is a degree zero divisor with support in $B$ such that all  $a_P$ are congruent modulo $p$ to a certain fixed (independent on $P$) integer $\mathbf{j}$.
Then
$$p^{r-1}D=p^{r-1}\mathbf{j}\left(\sum_{P\in B}(P)\right)+p^{r-1}p \left(\sum_{P\in B}\frac{(a_P-\mathbf{j})}{p}(P)\right)=$$
$$p^{r-1}\mathbf{j}\left(\sum_{P\in B}(P)\right)+p^r \left(\sum_{P\in B}b_P(P)\right)$$
where $b_P=(a_P-\mathbf{j})/p$. Clearly,
$$\sum_{P\in B}b_P=\sum_{P\in B}\frac{(a_P-\mathbf{j})}{p}=\frac{1}{p}\left(\sum_{P\in B}(a_P-\mathbf{j})\right)=\frac{1}{p} n\left(-\mathbf{j}\right)=-p^{r-1}m \mathbf{j}.$$
This implies that
$$p^{r-1}D=p^{r-1}\mathbf{j}\left(\left(\sum_{P\in B}(P)\right)-n(Q)\right)+p^{r-1}\mathbf{j}n(Q)+p^r \left(\sum_{P\in B}b_P(P)\right)=$$
$$p^{r-1}\mathbf{j}\left(\left(\sum_{P\in B}(P)\right)-n(Q)\right)+p^r D_0$$
where $Q$ is any point of $B$ and
$$D_0=p^{r-1}\mathbf{j} m(Q)+\left(\sum_{P\in B}b_P(P)\right).$$
Since $\deg(D)=0$, the degree of $D_0$ is also zero. It follows from Corollary \ref{principalC} that the class of $p^{r-1}D$ is $0$.
\end{proof}

\end{document}